\documentclass[12pt]{amsart} 

\usepackage{custom_preamble,tikz}
\usetikzlibrary{arrows}

\usepackage{xcolor}
\usepackage[normalem]{ulem}

\newcommand{\old}[1]{\textcolor{red}{\sout{#1}}}
\newcommand{\new}[1]{\textcolor{blue}{#1}}

\newenvironment{oldversion}
  {\color{red}}
  {}

\newenvironment{newversion}
  {\color{blue}}
  {}

\begin{document}

\title{Coset decision trees and the Fourier algebra}

\author{\tsname}
\address{\tsaddress}
\email{\tsemail}

\maketitle

\section{Introduction}\label{sec.introduction}

The work of this paper can be viewed in three ways:
\begin{itemize}
\item as a relationship between Boolean functions with small spectral norm and certain decision trees;
\item as a description of integer-valued functions in the Fourier algebra of a finite group;
\item as an example of how to extend certain additive combinatorial results to a non-abelian setting.
\end{itemize}
The books \cite{odo::1}, \cite{rud::1} and \cite{taovu::} provide background for each of these three perspectives respectively.

We discuss the first two in order in the introduction, and those interested in the third can move to \S\ref{sec.over} (though it may be worth noting a few basic definitions from \S\ref{sec.not} first).

The purpose of the paper is to develop various arguments which currently exist for abelian groups in a more general setting.  It may well be worth first consulting some of the abelian material, for example the paper \cite{shptalvol::1} which gives a good introduction from the perspective of computer science, or \cite{gresan::0} for the additive combinatorial viewpoint.

Given a finite abelian group $G$ we write $\wh{G}$ for its dual group, that is the set of homomorphisms $G \rightarrow S^1$ where $S^1:=\{z \in \C: |z|=1\}$.  The \textbf{Fourier transform} and \textbf{algebra norm} (sometimes called the \textbf{spectral norm}) of $f:G \rightarrow \C$ are defined by
\begin{equation}\label{eqn.sn}
\wh{f}:\wh{G} \rightarrow \C; \gamma\mapsto \E_{x \in G}{f(x)\overline{\gamma(x)}} \text{ and } \|f\|_{A(G)}:=\|\wh{f}\|_{\ell_1(\wh{G})}=\sum_\gamma{|\wh{f}(\gamma)|}.
\end{equation}

The paper \cite{kusman::} was one of the early papers to study the class of Boolean functions with small algebra norm, and amongst other things they showed that such functions can be efficiently learnt both randomly \cite[Theorem 4.2]{kusman::} and deterministically \cite[Theorem 4.12]{kusman::}.  Their arguments are based around algorithms for finding significant Fourier coefficients and these have been generalised (from their setting of $G=(\Z/k\Z)^n$) to general finite abelian groups in \cite{aka::0} leading to analogous learning results.  (The case of cyclic groups is already rather different to the situation of cubes considered in \cite{kusman::}.  In this case a deterministic  learning algorithm complementing \cite[Theorem 4.12]{kusman::} is given in \cite[Theorem 1]{aka::3}, and using some explicit constructions also developed in \cite{aka::2} and \cite{aka::1}.) There have also been extensions to (possibly non-abelian) $p$-groups\footnote{Boneh notes all such groups are nilpotent and hence monomial groups which helps with their representation theory.} by Boneh in \cite{bon::0} which can be used to produce a randomised efficient learning algorithm in that setting. 

In potentially non-abelian groups we need to give a slightly different definition of the algebra norm.  One can replace $\wh{G}$ above by the set of irreducible representations and extend the definitions in (\ref{eqn.sn}).  This is the approach taken in \cite[\S3]{bon::0} but we shall proceed slightly differently avoiding any representation theory.

Given a finite group $G$, any $f:G \rightarrow \C$ naturally induces a linear operator
\begin{equation*}
L_2(G) \rightarrow L_2(G); g \mapsto \left(x\mapsto f \ast g(x):=\E_{y\in G}{f(y)g(y^{-1}x)}\right),
\end{equation*}
and we define the \textbf{algebra norm} of $f$, written $\|f\|_{A(G)}$, to be the trace norm of this operator or, equivalently, the sum of its singular values\footnote{For a definition (of what are there called singular numbers) see \cite[Chapter III.G, \S6]{woj::}.}.  This norm makes the space of complex-valued functions on $G$ into a complex Banach algebra.

When $G$ is abelian the Fourier transform gives a unitary map $L_2(G) \rightarrow \ell_2(\wh{G})$ (simultaneously) diagonalising these operators.  The values on the diagonal are just the Fourier coefficients $\wh{f}(\gamma)$ and so the trace norm is exactly $\|\wh{f}\|_{\ell_1(\wh{G})}$ and our new definition agrees with (\ref{eqn.sn}).

Returning to \cite{kusman::}, Kushilevitz and Mansour also showed that Boolean functions that can be computed by small decision trees have small algebra norm.  To explain this it will be helpful to have a little more notation.

Suppose that we have a set $\mathcal{W}$ of subsets of $X$ where for each $W \in \mathcal{W}$ and $x \in X$ it is cheap to determine whether $x \in W$.  We define a \textbf{(binary) $\mathcal{W}$-decision tree} to be a rooted binary tree in which each leaf is labelled with $0$ or $1$; each internal vertex $v$ is labelled with some $W_v\in \mathcal{W}$; and the two outgoing edges of $v$ with $0$ and $1$.  Given $x \in X$, the decision tree $T$ constructs a computation path from the root to a leaf: when the path reaches vertex $v$ it follows the outgoing edge labelled $1$ if $x\in W_v$ and $0$ otherwise.  The output of $T$ on input $x$ is the value of the leaf. (\emph{c.f.} \cite[Definition 3.13]{odo::1}.)

An example of a $\mathcal{W}$-decision tree computing the Boolean function
\begin{equation*}
f=1_{W_0}1_{W_1}1_{W_3} + 1_{W_0}(1-1_{W_1}) + (1-1_{W_0})(1-1_{W_2})1_{W_5} \text{ where } W_0,\dots,W_5 \in \mathcal{W}
\end{equation*}
is given in Figure \ref{fig.example}.  (Of course it would be more efficient to replace the $W_4$ node since both its out-edges are leaves with value $0$; we leave it in for illustrative purposes.)
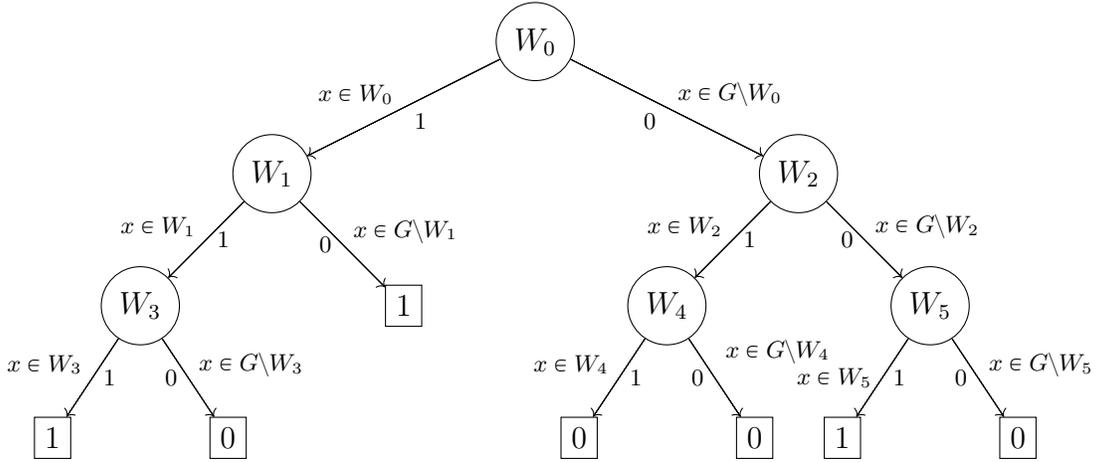
\begin{figure}
\centering
\begin{tikzpicture}[->,level/.style={sibling distance=70mm/#1},level distance=50pt]
\node [circle,draw] {$W_0$}
  child {node [circle,draw]  {$W_1$}
    child {node [circle,draw]  {$W_3$}
      child {node [rectangle,draw] {1}
            edge from parent
    node[left, yshift=5pt] {\scriptsize{$x \in W_3$}}
    edge from parent
    node[right] {\scriptsize{$1$}}}
      child {node [rectangle,draw] {0}
            edge from parent
    node[right, yshift=5pt] {\scriptsize{$x \in G\setminus W_3$}}
    edge from parent
    node[left] {\scriptsize{$0$}}}
          edge from parent
    node[left, yshift=5pt] {\scriptsize{$x \in W_1$}}
    edge from parent
    node[right] {\scriptsize{$1$}}
    }
  child {node [rectangle,draw] {$1$}
      edge from parent
    node[right, yshift=5pt] {\scriptsize{$x \in G\setminus W_1$}}
    edge from parent
    node[left] {\scriptsize{$0$}}
    }
    edge from parent
    node[left, yshift=5pt] {\scriptsize{$x \in W_0$}}
    edge from parent
    node[right,yshift=-5pt] {\scriptsize{$1$}}
  }
  child {node [circle,draw]  {$W_2$}
    child {node [circle,draw]  {$W_4$}
      child {node [rectangle,draw] {$0$}
                edge from parent
    node[left, yshift=5pt] {\scriptsize{$x \in W_4$}}
    edge from parent
    node[right] {\scriptsize{$1$}}}
      child {node [rectangle,draw] {$0$}
                edge from parent
    node[right, yshift=10pt] {\scriptsize{$x \in G\setminus W_4$}}
    edge from parent
    node[left] {\scriptsize{$0$}}}
    edge from parent
    node[left, yshift=5pt] {\scriptsize{$x \in W_2$}}
    edge from parent
    node[right] {\scriptsize{$1$}}
    }
      child {node [circle,draw] {$W_5$}
      child {node [rectangle,draw] {1}
                edge from parent
    node[left] {\scriptsize{$x \in W_5$}}
    edge from parent
    node[right] {\scriptsize{$1$}}}
      child {node [rectangle,draw] {0}
                edge from parent
    node[right,yshift=5pt] {\scriptsize{$x \in G\setminus W_5$}}
    edge from parent
    node[left] {\scriptsize{$0$}}}
        edge from parent
    node[right, yshift=5pt] {\scriptsize{$x \in G\setminus W_2$}}
    edge from parent
    node[left] {\scriptsize{$0$}}
  }
      edge from parent
    node[right, yshift=5pt] {\scriptsize{$x \in G\setminus W_0$}}
    edge from parent
    node[left,yshift=-5pt] {\scriptsize{$0$}}
};
\end{tikzpicture}
\caption{Example of a $\mathcal{W}$-decision tree.} \label{fig.example}
\end{figure}
The idea is, of course, that if the functions in $\mathcal{W}$ are easy to compute and $f$ can be computed by a small $\mathcal{W}$-decision tree, then $f$ is easy to compute.

Suppose that $T$ is a $\mathcal{W}$-decision tree computing $f:G \rightarrow \{0,1\}$.  If $P=v_0\cdots v_r$ is a maximal path in $T$ then we write $z_P$ for the value of the label on the leaf in $P$; and for $0 \leq i <r$ we write $g_i:=1_{W_{v_i}}$ if the edge $v_iv_{i+1}$ is labelled with a $1$ and $g_i:=1-1_{W_{v_i}}$ if it is labelled by a $0$; we define $g_P$ to be the product $g_0\cdots g_{r-1}$. Then (\emph{c.f.} \cite[Fact 3.15]{odo::1})
\begin{equation}\label{eqn.id}
f(x)=\sum_{P\text{ is a maximal path in }T}{z_Pg_P(x)} \text{ for all }x \in G.
\end{equation}

A \textbf{parity decision tree}\footnote{See \cite[Exercise 3.26]{odo::1}.} corresponds to the case
\begin{equation*}
X=(\Z/2\Z)^n \text{ and }\mathcal{W}=\{H \leq X: |X:H|=2\},
\end{equation*}
and a \textbf{decision tree}\footnote{See \cite[Definition 3.13]{odo::1}.} corresponds to the case
\begin{equation*}
X=(\Z/2\Z)^n\text{ and }\mathcal{W}=\{\{x:x_i=0\}: 1 \leq i \leq n\}.
\end{equation*}
There are also notions of $k$-ary decision trees which are natural when $X=(\Z/k\Z)^n$, and \cite[Definition 5.1]{bon::0} defines a $G$-decision tree in which each internal vertex has a normal subgroup $H \lhd G$ associated with it and an out-edge for each coset of $H$.

In \cite[Lemma 5.1]{kusman::} the authors show that if a Boolean function can be computed by a parity decision tree with $m$ leaves then it has algebra norm at most $m$.  Given this it is natural to ask whether every Boolean function on $(\Z/2\Z)^n$ with small spectral norm can be computed by a parity decision tree with a small number of leaves.  Unfortunately this is not quite true: one can check that if $f$ is not identically $0$ and a Boolean function computed by a parity decision tree with $m$ leaves then\footnote{See \cite[Exercise 3.30]{odo::1} for the decision tree version of this.} $\E{f} \geq 2^{-m}$.  But if $f$ is the indicator function of a singleton then $\|f\|_{A(G)}=1$, so it follows that $f$ is a Boolean function with algebra norm $1$ that cannot be computed by a parity decision tree with $o(n)$ leaves.

The singleton example in the previous paragraph extends to a general class of examples for finite groups: if $G$ is a finite group then we write $\mathcal{W}(G)$ for the set of cosets of subgroups of $G$.  A short calculation shows that $\|1_W\|_{A(G)}=1$.

We call a $\mathcal{W}(G)$-decision tree a \textbf{coset decision tree}, and then in view of (\ref{eqn.id}) and the fact that the algebra norm really is an algebra norm we see that if $f:G \rightarrow \{0,1\}$ can be computed by a coset decision tree with $m$ leaves then
\begin{equation*}
\|f\|_{A(G)} \leq m2^m.
\end{equation*}
In this paper we shall show the following converse.
\begin{theorem}\label{thm.mn}
Suppose that $G$ is a finite group and $f:G \rightarrow \{0,1\}$ has $\|f\|_{A(G)}\leq M$.  Then there is a coset decision tree with $\exp(\exp(\exp(O(M^{2}))))$ leaves computing $f$.
\end{theorem}
We remark that if $G$ is abelian then better results are available.  Indeed, if $G=(\Z/p\Z)^n$, Shpilka, Tal and lee Volk \cite[Theorem 1.2]{shptalvol::1} identify the stronger structure of a parity decision tree with far better bounds.  (Of course their bounds necessarily depend on the size of $G$, but they do so in a very mild way, and their theorem comes along with a host of other results.)

We now turn to our second, harmonic analytic, motivation.  In \cite[Definition 3.5]{eym::} Eymard extended the classical definition\footnote{See, for example,  \cite[\S1.2.3]{rud::1}.} of $A(G)$ for locally compact abelian groups to locally compact groups.  This is done first by extending the Fourier-Stieltjes algebra $B(G)$ \cite[Definition 2.2]{eym::}.  

Given a locally compact group $G$ and function $f:G \rightarrow \C$ define
\begin{equation}\label{eqn.norm}
\|f\|_{B(G)}:=\inf\{\|x\|_H\|y\|_H:f(t)=\langle \pi(t)x,y\rangle_H \text{ for all }t \in G\},
\end{equation}
where the infimum is over all Hilbert spaces $H$; elements $x,y \in H$; and strongly continuous representations $\pi:G \rightarrow \Aut(H)$ -- that is strongly continuous unitary representations of $G$ on $H$.  The set $B(G)$ is then the set of functions for which this quantity is finite.

\cite[Proposition 2.16]{eym::} shows that $B(G)$ equipped with $\|\cdot\|_{B(G)}$ is a complex Banach algebra and \cite[Lemme 2.14]{eym::} that the infimum in (\ref{eqn.norm}) is attained.  The space $A(G)$ can be defined as the closure of $B(G)\cap C_c(G)$ in $B(G)$ where $C_c(G)$ is the set of continuous compactly supported functions, and we write $\|f\|_{A(G)}:=\|f\|_{B(G)}$ for $f \in A(G)$.  If $G$ is finite then this definition agrees with the one in the previous section.

We write $\mathcal{W}(G)$ for the set of cosets of open subgroups of $G$ and a short calculation confirms that if $W \in \mathcal{W}(G)$ then $\|1_W\|_{B(G)} =1$.  Moreover, if $z \in \ell_1(\mathcal{W}(G))$ is integer-valued then
\begin{equation*}
f:=\sum_{W \in \mathcal{W}(G)}{z_W1_W} \text{ is integer-valued and has }\|f\|_{B(G)} \leq \|z\|_{\ell_1(\mathcal{W}(G))}.
\end{equation*}
In \cite{lef::}, Lefranc announced the following converse.
\begin{theorem}\label{thm.lf}
Suppose that $G$ is a locally compact group and $f \in B(G)$ is integer-valued.  Then there is some integer-valued $z \in \ell_1(\mathcal{W}(G))$ such that
\begin{equation*}
f=\sum_{W \in \mathcal{W}(G)}{z_W1_W}.
\end{equation*}
\end{theorem}
It seems Lefranc's proof of Theorem \ref{thm.lf} never appeared in print, but happily a beautiful argument of Host did in \cite{hos::}, and Theorem \ref{thm.lf} is sometimes (see \emph{e.g.} \cite{run::}) called the Cohen-Host theorem since Cohen \cite{coh::} proved it in the case when $G$ is abelian.

It has been used for a number of endeavours in harmonic analysis, for example characterising the locally compact groups $G$ for which $A(G)$ is amenable \cite{forrun::}; and characterising the ideals of $A(G)$ with bounded approximate identities \cite{forkanlauspr::}.  Both of these and more are discussed in \cite{run::}.

Host's argument actually shows the following stronger result.
\begin{theorem}\label{thm.host}
Suppose that $G$ is a locally compact group and $f \in B(G)$ is integer-valued with $\|f\|_{B(G)}\leq M$.  Then there is an integer $L$\old{$\leq M$}\new{$\leq M^2$}, open subgroups $K_1,\dots,K_L$ of $G$, and integer-valued functions $z^{(i)} \in \ell_1(G/K_i)$ such that
\begin{equation*}
f=\sum_{i=1}^L{\sum_{W \in G/K_i}{z_W^{(i)}1_W}}.
\end{equation*}
\end{theorem}
The bound on $L$ makes this result partially quantitative, but it still tells us nothing if $G$ is a finite group, and we prove the following.
\begin{theorem}\label{thm.mn2}
Suppose that $G$ is a finite group and $f:G\rightarrow \Z$ has $\|f\|_{A(G)} \leq M$.  Then there are $L=O(M^2)$ and $R=\exp(\exp(\exp(O(M^2))))$, subgroups $K_1,\dots,K_L \leq G$, and integer-valued functions $z^{(i)} \in \ell_1(G/K_i)$ (for $1\leq i \leq L$) such that
\begin{equation*}
f=\sum_{i=1}^L{\sum_{W \in G/K_i}{z_W^{(i)}1_W}} \text{ and } \|z^{(i)}\|_{\ell_1(G/K_i)} \leq R.
\end{equation*}
\end{theorem}
This improves on the bounds in \cite[Theorem 1.2]{san::9} (which are triply tower in nature), and we also hope it provides a more easily digested proof.

If $f$ is the indicator function of an arithmetic progression of integers with size $N$ then its algebra norm was computed classically (see \emph{e.g.} \cite[(17.)]{fej::0}) and shown to be $\frac{4}{\pi^2}\log N +O(1)$.  Embedding this progression in $\Z/p\Z$ for a sufficiently large prime $p$ yields a subset of a finite group of size $N$ with algebra norm asymptotic to $\frac{4}{\pi^2}\log N+O(1)$.  Since the only subgroups of $\Z/p\Z$ are the whole group and the trivial group we see that $R =\Omega(\exp(\frac{\pi^2}{4}M))$ in Theorem \ref{thm.mn2}.

In the setting above -- small sets in cyclic groups of prime order -- Theorem \ref{thm.mn2} is known with bounds of the form $R=\exp(O(M))$ (see \cite[Theorem 2]{konshk::1}), so in some sense the example above is tight.  It seems plausible that arithmetic progressions are the worst examples more generally.  In the dyadic groups -- groups of the form $\F_2^n$ -- we do not even have arithmetic progressions and one might hope there that a polynomial bound for $R$ holds.  We do not know of examples that make fundamental use of the non-abelian structure of more general groups.

As a final remark it might be of interest to combine the arguments here with Host's to prove a fully quantitative version of Theorem \ref{thm.host}.  Indeed, when $G$ is abelian this was done in \cite[Theorem 1.2]{gresan::0} and that result has since been applied in \cite{woj::0} and \cite{czuwoj::}.

\section{Notation and basic facts}\label{sec.not}

It is convenient to gather together a few definitions and standard lemmas along with some alternative ways of defining the algebra norm.  We assume from now on that $G$ is a finite group.

We write $C(G)$ for the complex-valued functions on $G$ and $\rho$ for the right regular representation so that
\begin{equation*}
\rho_y(f)(x):=f(xy) \text{ for all }x,y \in G.
\end{equation*}
We write $M(G)$ for the complex-valued measures on $G$ and extend $\rho$ to $M(G)$ in the natural way.  Moreover we put
\begin{equation*}
f\ast \mu(x):=\int{\rho_{y^{-1}}(f)(x)d\mu(y)} \text{ for all }f \in C(G) \text{ and }\mu \in M(G),
\end{equation*}
so that
\begin{equation*}
\rho_z(f \ast \mu) = \int{\rho_{zy^{-1}}(f)d\mu(y)} = \int{\rho_{u^{-1}}(f)d\rho_z(\mu)(u)} = f \ast \rho_z(\mu) \text{ for all }z \in G.
\end{equation*}
We define convolution of measures similarly.

If $\mu \in M(G)$ is non-negative and $g \in L_1(\mu)$ then we write $gd\mu$ for the element of $M(G)$ induced by
\begin{equation*}
C(G) \rightarrow \C; f \mapsto \int{fgd\mu}.
\end{equation*}

If $S \subset G$ is non-empty then we write $m_S$ for the uniform probability measure on $G$ supported on $S$, and $\delta_S$ for the counting measure supported on $S$; we write $\ell_p(S)$ for $L_p(\delta_S)$.

If $\mu$ is a Haar measure on $G$ then we define
\begin{equation*}
f \ast g:=f \ast (gd\mu) \text{ for all } f \in L_1(\mu) \text{ and }g \in L_1(\mu).
\end{equation*}
The two examples we use are when $f,g \in L_1(m_G)$ and $h,k \in \ell_1(G)$ when we have
\begin{equation*}
f \ast g=\E_y{\rho_{y^{-1}}(f)g(y)} \text{ and } h \ast k = \sum_y{\rho_{y^{-1}}(h)k(y)}.
\end{equation*}
We also put
\begin{equation*}
\wt{f}(x):=\overline{f(x^{-1})} \text{ for all }x \in G, f \in C(G),
\end{equation*}
so that $\wt{1_A} = 1_{A^{-1}}$ for any $A \subset G$, and make a similar definition for $\wt{\mu}$ where $\mu \in M(G)$.

For $f \in C(G)$ and $\mu \in M(G)$ we write
\begin{equation*}
\langle f,\mu\rangle:=\int{fd\overline{\mu}} \text{ and } \langle \mu,f\rangle:=\int{\overline{f}d\mu}.
\end{equation*}
A short calculation verifies that
\begin{equation*}
\langle f,\nu\ast \wt{\mu}\rangle=\langle f\ast \mu,\nu\rangle = \langle \mu, \wt{f}\ast \nu\rangle \text{ for all }f \in C(G), \mu,\nu \in M(G).
\end{equation*}
Given a homomorphism $\pi:G \rightarrow \Aut(H)$ we write
\begin{equation*}
\wh{\mu}(\pi):=\int{\pi(t^{-1})d\mu(t)}.
\end{equation*}
This is the analogue of the Fourier transform.
\begin{oldversion}\begin{lemma}\label{lem.split}
Suppose that $H \leq G$ and $f \in A(G)$.  Then
\begin{equation*}
\|f\|_{A(G)} = \|f-f\ast m_H\|_{A(G)} + \|f\ast m_H\|_{A(G)}.
\end{equation*}
\end{lemma}
\begin{proof}
This is a routine calculation: write $M$ for the operator $g \mapsto f \ast g -f \ast m_H \ast g$ and $N$ for the operator $g \mapsto f \ast m_H \ast g$.  Then $N^*g=m_H \ast \wt{f}\ast g$ and so $MN^*g = 0$.  But then $M^*MN^*N = 0=N^*NM^*M$ and so $M^*M$ and $N^*N$ commute.  By the spectral theorem they can be simultaneously diagonalised and any eigenvector for $M^*M$ with non-zero eigenvalue is in the kernel of $N^*N$ and vice versa.  The result follows from the definition of the $A(G)$-norm.
\end{proof}
\end{oldversion}
\begin{newversion}
This lemma is false. My thanks to Daniel Naylor for alerting me to this and also showing me an example forcing $L=\Omega(M^2)$ in Theorem \ref{thm.mn2} which, in turn, motivated the update to Theorem \ref{thm.host}.

 The error in the proof is that although $MN^*=0$, the same is not true of $NM^*$ in general. A counterexample to the conclusion is given by $G=S_3$, $H \leq G$ of size $2$, and $A=H \cup Hx$ for some $x \not \in H$. Then  $\|1_A \ast m_H\|_{A(G)} \geq 1$ and $\|1_A - 1_A\ast m_H\|_{A(G)} \geq \frac{1}{2}$. However, $\|1_A\|_{A(G)}=\frac{4}{3}<1+\frac{1}{2}$.

There is an easy fix due to Beke \emph{et al.\ }\cite[Lemma 4.2]{Beke:2026aa}. In that paper a much stronger version of Theorem \ref{thm.mn2} due to Carenini is also proved and recorded in \cite[Theorem 4.1]{Beke:2026aa}.
 \begin{lemma*}[Lemma \ref{lem.split}, fixed]
Suppose that $G$ is a finite group, $H \leq G$ and $f \in A(G)$.  Then
\begin{equation*}
\|f\|_{A(G)}^2 \geq \|f-f\ast m_H\|_{A(G)}^2 + \|f \ast m_H\|_{A(G)}^2.
\end{equation*}
\end{lemma*}
\begin{proof}
Write $M$ for the operator $g \mapsto f \ast g$ and $P$ for the orthogonal projection operator $g \mapsto m_H \ast g$. The inequality we want is the reverse Minkowski inequality (\emph{c.f.\ }\cite[165, p.~123]{harlitpol::}) for the map $T \mapsto (\tr |T|^r)^{\frac1r}$ with $r\in (0,1)$, applied with $r=\frac12$ to $MM^*=M(I-P)M^* + MPM^*$.

Let $U,V$ have operator norm at most $1$ and be such that
\begin{equation*}
\alpha:=\|f - f \ast m_H\|_{A(G)}=\tr M(I-P)U\text{ and }\beta:=\|f \ast m_H\|_{A(G)}=\tr MPV.
\end{equation*}
We may suppose that $U=(I-P)U$ and $V=PV$. Hence $U^*V=U^*PV=(PU)^*V=0$ and $V^*U=V^*(I-P)U=((I-P)V)^*U=0$. Thus
\begin{equation*}
W:=\frac{\alpha}{\sqrt{\alpha^2+\beta^2}}U + \frac{\beta}{\sqrt{\alpha^2+\beta^2}}V
\end{equation*}
has
\begin{align*}
\langle W^*Wv,v\rangle& =\frac{\alpha^2\langle U^*Uv,v\rangle}{\alpha^2+\beta^2} +\frac{\alpha \beta}{\alpha^2+\beta^2}(\langle V^*Uv,v\rangle + \langle U^*Vv,v\rangle)+\frac{\beta^2\langle V^*Vv,v\rangle}{\alpha^2+\beta^2}\\ & =\frac{\alpha^2}{\alpha^2+\beta^2}\|Uv\|^2+\frac{\beta^2}{\alpha^2+\beta^2}\|Vv\|^2\leq \|v\|^2.
\end{align*}
It follows that $W$ has operator norm at most $1$, and so $|\tr MW| \leq \|f\|_{A(G)}$, while
\begin{align*}
\tr MW &= \frac{\alpha}{\sqrt{\alpha^2+\beta^2}}\tr MU +\frac{\beta}{\sqrt{\alpha^2+\beta^2}}\tr MV \\ & = \sqrt{\alpha^2+\beta^2}.
\end{align*}
Rearranging gives the result.
\end{proof}
\end{newversion}

We now turn to two useful equivalent definitions for the algebra norm.  The first can be compared with the definition of the uniform almost periodicity norms \cite[Definition 11.15]{taovu::}, but see also \cite[Th{\'e}or{\`e}m, p.~218]{eym::}.
\begin{lemma}\label{lem.various}
Suppose that $f \in A(G)$.  Then there is a constant $M \leq \|f\|_{A(G)}$; a finite probability space $(\Omega,\P)$; and functions $g_\omega, h_\omega \in L_2(m_G)$ of unit $L_2(m_G)$-norm for all $\omega \in \Omega$ such that
\begin{equation*}
f(x)=M\E_\omega{\wt{h_\omega} \ast g_\omega(x)} \text{ for all }x \in G.
\end{equation*}
There is also a finite-dimensional Hilbert space $H$; a homomorphism $\pi:G \rightarrow \Aut(H)$; and elements $v,w \in H$ such that $\|v\|\|w\| \leq \|f\|_{A(G)}$ and
\begin{equation*}
f(x)=\langle \pi(x)v,w\rangle \text{ for all }x \in G.
\end{equation*}
\end{lemma}
\begin{proof}
Let $(\lambda_\omega)_{\omega}$ be the singular values of $g\mapsto f \ast g$ with corresponding bases $(v_\omega)_{\omega}$ and $(w_\omega)_{\omega}$ so that $f\ast v_\omega = \lambda_\omega w_\omega$.  Put $u:=|G|1_{\{1_G\}}$, whence
\begin{equation*}
f(x)=\int{\rho_{t^{-1}}(f \ast \rho_t(u)) (x)dm_G(t)} = \sum_\omega{\lambda_\omega\int{\rho_{t^{-1}}(w_\omega)\langle \rho_t(u),v_\omega\rangle dm_G(t)}}.
\end{equation*}
Let $\P(\omega)=\lambda_\omega\|f\|_{A(G)}^{-1}$ (which is a probability measure by definition of $\|\cdot \|_{A(G)}$) and put $h_\omega:=\wt{w_\omega}$ and $g_\omega(t):=\langle \rho_t(u),v_\omega\rangle$.  It is then easy to check that $\|h_\omega\|_{L_2(m_G)} = \|w_\omega\|_{L_2(m_G)} = 1$ for all $\omega$; and furthermore
\begin{equation*}
\|g_\omega\|_{L_2(m_G)}^2 = \int{|\langle \rho_t(u),v_\omega\rangle|^2 dm_G(t)}=\|v_\omega\|_{L_2(m_G)}^2 = 1
\end{equation*}
for all $\omega$.  The first part follows.

Given the first representation note that
\begin{equation*}
\wt{h_\omega} \ast g_\omega(x) = \langle g_\omega,\rho_x(h_\omega)\rangle_{L_2(m_G)}.
\end{equation*}
Let $H:=L_2(m_G)^{\Omega}$ and $\rho$ the right regular representation in the natural way.  The result then follows from a calculation.
\end{proof}

It is useful to make a further definition: we write $\mathcal{N}(G)$ for the set of symmetric neighbourhoods of the identity in $G$ and for $Z,Z^+,Z^- \subset G$ non-empty and $X \in \mathcal{N}(G)$ we say that $(Z,X;Z^+,Z^-)$ is \textbf{$\eta$-closed} if
\begin{equation*}
Z^-X \subset Z , ZX^{-1} \subset Z^+ \text{ and } \frac{|Z^+\setminus Z^-|}{|Z|} \leq \eta.
\end{equation*}
Frequently we shall simply say $(Z,X)$ is $\eta$-closed and introduce $Z^+$ and $Z^-$ as needed.

Such sets support an approximate invariant measure captured by the following lemma.  The proof is immediate.
\begin{lemma}\label{lem.inv}
Suppose that $(Z,X)$ is $\eta$-closed.  Then
\begin{equation*}
\sup{\{|\rho_x(f\ast m_Z)(t) - f \ast m_Z(t)|: x \in X\}}=\|f\ast m_Z-f \ast m_Z(t)\|_{L_\infty(tX)} \leq \eta \|f\|_{L_\infty(G)}.
\end{equation*}
\end{lemma}
%%\begin{proof}
%%Suppose $x \in X$.  Then if $u \in Z^-$ we have $u,ux \in Z$; and if $u\not \in Z^+$ we have $u \not \in Z$ and $ux \not \in Z$.  Hence
%%\begin{align*}
%%\left|f \ast m_Z(tx) - f \ast m_Z(t)\right| & = \left|\int{f(txy^{-1})dm_Z(y)} - \int{f(ty^{-1})dm_Z(y)}\right|\\
%%& = \left|\int{f(tu^{-1})dm_{Z}(ux)} -\int{f(tu^{-1})dm_Z(u)}\right|\\ & \leq \|f\|_{L_\infty(G)}\frac{|Z^+ \setminus Z^-|}{|Z|} \leq \eta \|f\|_{L_\infty(G)}
%%\end{align*}
%%as required.
%%\end{proof}
The representation $\rho$ is isometric on $L_p(G)$ and we preserve an approximate version of this in the following facts:
\begin{equation*}
 \|\rho_x(f)\|_{L_p(m_{BX^{-1}})} = \|f\|_{L_p(m_{BX^{-1}})} \text{ whenever }f \in L_p(m_B),
\end{equation*}
and
\begin{equation*}
\|\rho_x(f)\|_{L_p(m_B)}^p\leq \frac{|BX|}{|B|}\|f\|_{L_p(m_{BX})}^p \text{ whenever } f \in L_p(m_{BX}).
\end{equation*}

Finally we shall need the following version of Ruzsa's covering lemma; the argument is the same as \cite[Lemma 2.14]{taovu::}.
\begin{lemma}[Ruzsa's covering lemma]\label{lem.rcl}
Suppose that $X,W \subset G$ are non-empty and have $D|W|\geq |WX|$.  Then there is a set $T \subset X$ of size at most $D$ such that $\{W^{-1}Wt:t \in T\}$ is a cover of $X$.
\end{lemma}
%%\begin{proof}
%%Let $T \subset X$ be maximal such that if $t,t' \in T$ are distinct then $Wt\cap Wt' = \emptyset$.  If $T' \subset T$ is finite then
%%\begin{equation*}
%%|T'||W| = \sum_{t \in T'}{|Wt|} = \left|\bigcup_{t \in T'}{Wt}\right| \leq |WX| \leq |WX|,
%%\end{equation*}
%%and so $|T| \leq D$.  If $x \in X$ then either $x \in T$ and, since $W$ is non-empty $1_G \in W^{-1}W$ and so $x \in W^{-1}Wx$; or $x \not \in T$ so there is some $t \in T$ such that $Wt \cap Wx \neq \emptyset$.  It follows that $x \in W^{-1}Wt$ and we are done.
%%\end{proof}

\section{Overview and conditional proof}\label{sec.over}

In the introduction we discussed two ways to view this work.  In this section we give an overview of the proof of our results and also highlight a third way to view our arguments, this time through an additive combinatorial lens.

The overall structure of our argument is fairly typical for additive combinatorics, translating the work of \cite{gresan::0} to the non-abelian setting.  This can be an involved affair as seen in \cite{san::9}, and part of our hope here is that the shorter and sharper arguments we now give will be more illuminating and useful.

The main approach is inductive but works over the class of functions whose values are \emph{almost} integers -- we say $f:G \rightarrow \C$ is \textbf{$\epsilon$-almost integer-valued} if there is some $f_\Z:G \rightarrow \Z$ such that $\|f-f_\Z\|_{L_\infty(G)} <\epsilon$.  If $\epsilon \in \left(0,\frac{1}{2}\right]$ then $f_\Z$ is uniquely defined and we shall always assume it is.

We shall prove the following result.
\begin{theorem}\label{thm.ky}
There is an absolute constant $C>0$ such that if $f$ is $\epsilon$-almost integer-valued, $\|f\|_{A(G)} \leq M$, and $\epsilon \leq \exp(-CM^2)$, then there is some $L=O(M^2)$, subgroups $H_1,\dots,H_L \leq G$, and integer-valued functions $z^{(i)} \in \ell_1(G/H_i)$ (for $1\leq i \leq L$) such that
\begin{equation*}
f_\Z=\sum_{i=1}^L{\sum_{W \in G/H_i}{z_W^{(i)}1_W}} \text{ and } \|z^{(i)}\|_{\ell_1(G/H_i)} \leq \exp(\exp(\exp(O(M^2)))).
\end{equation*}
\end{theorem}
First note that Theorem \ref{thm.mn2} follows immediately.

In thinking about this statement it may help to note that since $z^{(i)}$ is integer-valued the upper bound on $\|z^{(i)}\|_{\ell_1(G/H_i)}$ yields an upper bound on the size of the support of $z^{(i)}$ -- that is on the number of cosets of $H_i$ where $f_{\Z}$ can have support.
\begin{proof}[Proof of Theorem \ref{thm.mn}]
The argument is purely a manipulation of notation and is better answered through a picture.

Since $f$ is Boolean we have $f=f_\Z$ and so $f$ is certainly $\exp(-CM^2)$-integer-valued; we can apply Theorem \ref{thm.ky} to get subgroups $H_1,\dots,H_L \leq G$ and integer-valued functions $z^{(i)} \in \ell_1(G/H_i)$ such that
\begin{equation*}
f=\sum_{i=1}^L{\sum_{W \in G/H_i}{z_W^{(i)}1_W}} \text{ and } \|z^{(i)}\|_{\ell_1(G/H_i)} \leq \exp(\exp(\exp(O(M^2)))).
\end{equation*}
Let $g_1^{(i)}H_i<\dots<g_{R_i}^{(i)}H_i$ be an arbitrary ordering of the $R_i=\exp(\exp(\exp(O(M^2))))$ cosets of $H_i$ in the support of $z^{(i)}$.  Construct a coset decision tree $T^{(i)}$ with a root at $g_1^{(i)}H_i$; an edge from $g_j^{(i)}H_i$ to $g_{j+1}^{(i)}H_i$ for all $j<R_i$ labelled $0$; and a leaf at every vertex labelled $1$.  (The tree is really a decision list\footnote{See \cite[Exercise 3.23]{odo::1}.} -- see Figure \ref{fig.2}.)
\begin{figure}
\centering
\begin{tikzpicture}[level/.style={level distance=95pt, sibling distance=50mm/#1}]
\node [circle,draw] {$g_1^{(i)}H_i$}
  child[grow=right] {node [circle,draw] {$g_2^{(i)}H_i$}
    child[grow=down] {node [rectangle,draw,yshift=30pt]  {$*$}
    edge from parent
    node[right, yshift=5pt] {\scriptsize{$x \in g_2^{(i)}H_i$}}
    node[left,yshift=5pt] {\scriptsize{$1$}}
      } 
   child[grow=right] {node {$\cdots$}
      child[grow=right] {node [circle,draw] {$g_{R_i}^{(i)}H_i$}
    child[grow=down] {node [rectangle,draw,yshift=30pt]  {$*$}
    edge from parent
    node[right, yshift=5pt] {\scriptsize{$x \in g_{R_i}^{(i)}H_i$}}
    node[left,yshift=5pt] {\scriptsize{$1$}}}
    child[grow=right] {node [rectangle,draw]  {$*$}
    edge from parent
    node[above, yshift=5pt] {\scriptsize{$x \in G\setminus g_{R_i}^{(i)}H_i$}}
    node[below] {\scriptsize{$0$}}
      } 
      edge from parent
    node[above, yshift=5pt] {\scriptsize{$x \in G\setminus g_{R_i-1}^{(i)}H_i$}}
    node[below] {\scriptsize{$0$}}
      }edge from parent
    node[above, yshift=5pt] {\scriptsize{$x \in G\setminus g_2^{(i)}H_i$}}
    node[below] {\scriptsize{$0$}}
    }
    edge from parent
    node[below] {\scriptsize{$0$}}
    node[above, yshift=5pt] {\scriptsize{$x \in G\setminus g_1^{(i)}H_i$}}
  }
  child[grow=down] {node [rectangle,draw,yshift=30pt]  {$*$}
  edge from parent
    node[right, yshift=5pt] {\scriptsize{$x \in g_1^{(i)}H_i$}}
    node[left,yshift=5pt] {\scriptsize{$1$}}
};
\end{tikzpicture}
\caption{The coset decision tree $T^{(i)}$ with asterisks where a root of one of $R_i+1$ copies of $T^{(i+1)}$ go.}\label{fig.2}
\end{figure}
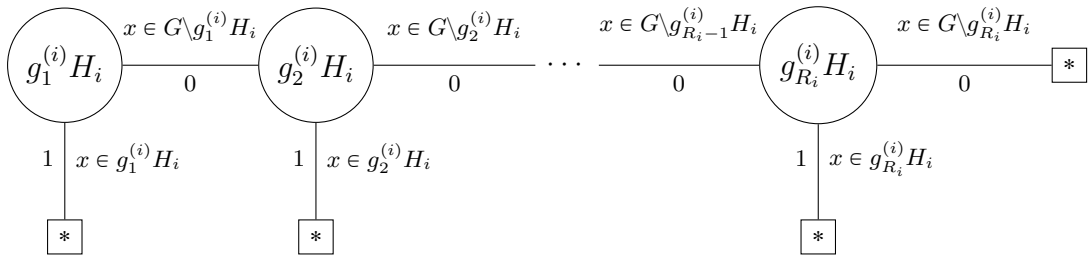

We produce $T$ iteratively and it is convenient to enlarge our class of coset decision trees to include integer values on the leaves not just values in $\{0,1\}$.  The final tree we produce has values in $\{0,1\}$.  We start with $T_1:=T^{(1)}$ and write $l_i$ for the leaf-value function on $T_i$.  At stage $i<L$ take every leaf and append a copy of $T^{(i+1)}$ such that $v$ is the root of $T^{(i+1)}$, copying all the edge values from $T^{(i+1)}$ in the obvious way.  Given a leaf $w$ in the new graph, let $w'$ be the vertex it is connected to and define
\begin{equation*}
l_{i+1}(w):=\begin{cases} l_i(v) + z_{g_j^{(i+1)}H_{i+1}}^{(i+1)}& \text{ if }w' \text{ is labelled }g_j^{(i+1)}H_{i+1} \text{ and }ww' \text{ has value }1\\
l_i(v) & \text{ otherwise.}
\end{cases}
\end{equation*}
We terminate with $T:=T_L$.  At the end of this process, suppose that we look at a computation path for $x \in G$.  This gives us a unique path from the root to some vertex $v$, say
\begin{equation}\label{eqn.path}
v_{1,1},\dots,v_{r_1,1},v_{1,2},\dots,v_{r_2,2},v_{1,3},\dots,v_{r_{L-1},L-1},v_{1,L},\dots,v_{r_L,L},v_{1,L+1}:=v,
\end{equation}
such that $v_{j,i}$ is labelled $g_j^{(i)}H_i$ for $1 \leq j \leq r_i$ and $1\leq i \leq L$.  The fact that all the cosets of $H_1$ occur before those of $H_2$ \emph{etc.} simply reflects the order we built up $T$.

Write $I \subset \{1,\dots,L\}$ for the set of indices such that the edge between $v_{r_i,i}$ and $v_{1,i+1}$ is a $1$.  If the value of the edge between $v_{r_i,i}$ and $v_{1,i+1}$ is $0$ then $r_i=R_i$ and $x \in G\setminus g_j^{(i)}H_i$ for all $1 \leq j \leq R_i$ and so $x \not \in \bigcup{\supp z^{(i)}}$.  If the value of the edge between $v_{r_i,i}$ and $v_{1,i+1}$ is $1$ then $x \in G\setminus g_j^{(i)}H_i$ for all $1 \leq j < r_i$ and $x \in g_{r_i}^{(i)}H_i$.  It follows that
\begin{equation*}
f(x)=\sum_{i=1}^L{\sum_{W \in G/H_i}{z_W^{(i)}1_W(x)}}=\sum_{i \in I}{z^{(i)}_{g_{r_i}^{(i)}H_i}}=l_L(v)
\end{equation*}
as required.  The total number of leaves of the resulting coset decision tree is at most $(R_1+1)\cdots (R_L+1) \leq \exp(\exp(\exp(O(M^2))))$ as claimed.
\end{proof}

To prove Theorem \ref{thm.ky} we need two key ingredients.  The first shows us how to find structure in the support of $f_\Z$ when $f$ has small algebra norm and is almost integer-valued.  We shall prove this in \S\ref{sec.ac}.
\begin{proposition}\label{prop.infstruct}
There is an absolute $C>0$ such that if $f$ is $\epsilon$-almost integer-valued with $\|f\|_{A(G)} \leq M$ and $\epsilon \leq \exp(-CM)$, then there is some $S \subset \supp f_\Z$ such that $|SS^{-1}| \leq M^{O(M)}|S|$ and $|S| = M^{-O(M)}|\supp f_\Z|$.
\end{proposition}
When $G$ is abelian the conclusion $|SS^{-1}| \leq K|S|$ (in additive notation $|S-S| \leq K|S|$) output above is the input for Fre{\u\i}man-type theorems.  The output of these \emph{e.g.} \cite[Theorem 5.46]{taovu::} is a set of the form $P(d_1,N_1)+\dots+P(d_r,N_r) +H$ where $H \leq G$ and $P(d_i,N_i)$ is an arithmetic progression of length $2N_i+1$ and common difference $d_i$ centred at $0_G$.  For us the important feature is that the sets $B_i:=P(d_1,2^{-i}N_1)+\dots+P(d_r,2^{-i}N_r) +H$ (for $i \in \N_0$) form a base for a topology on $G$ with some nice properties. In particular, all the sets $B_i$ are symmetric neighbourhoods of the identity with
\begin{equation}\label{eqn.K}
B_{i+1} + B_{i+1} \subset B_i \text{ and } |B_{i+1}| = \Omega_K(|B_i|) \text{ for all }i \in\N_0
\end{equation}
and
\begin{equation*}
|B_0| = \Omega_K(|A|) \text{ and }B_0 \subset 2A-2A.
\end{equation*}
The fact that the $\Omega$-term in (\ref{eqn.K}) does not depend on $i$ captures the linear structure of characters, and replicating this is a key hurdle in proving a Fre{\u\i}man-type theorem in general groups.  This (and much more) has now been achieved by Breuillard, Green and Tao in \cite[Theorem 1.6]{bregretao::0}, but at the cost of weaker dependencies.  We take a different approach and accept some (relatively) mild $i$ dependence in exchange for better $K$-dependence.
\begin{lemma}\label{lem.fr}
Suppose that $A$ is non-empty and $|AA^{-1}| \leq K|A|$ and $\eta \in (0,1]$.  Then there are $Z,Y \in \mathcal{N}(G)$ such that $(Z,Y^4)$ is $\eta$-closed with 
\begin{equation*}
|Y| \geq \exp(-O(\eta^{-2}\log^{O(1)}2K))|A|\text{ and } m_{A^{-1}} \ast 1_{AA^{-1}} \ast m_{A}(x) > \frac{1}{2} \text{ for all }x \in (Z^+)^4.
\end{equation*}
In particular, $(Z^+)^4 \subset A^{-1}AA^{-1}A$.
\end{lemma}
We prove this in \S\ref{sec.tools}.

This lemma is designed to be used iteratively and the conclusion in terms of the convolution is there to deal with the first step when we may know that $|AA^{-1}| \leq K|A|$ but not that $|A^{-1}AA^{-1}A| =O_K(|A|)$.  (In the abelian setting Pl{\"u}nnecke's inequality \cite[Corollary 6.26]{taovu::} gives the latter as a consequence of the former but in non-abelian groups this need not be the case.  See the discussion after \cite[Proposition 2.38]{taovu::}.)

As well as providing us with a way of dealing with the output of Proposition \ref{prop.infstruct}, Lemma \ref{lem.fr} also provides us with a general way of replacing level sets of characters -- Bohr sets\footnote{See \cite[\S4.4]{taovu::}.} -- in the abelian setting.  (The obvious analogue of using characters or representations of non-abelian groups runs into complications with controlling the dimension of the representation.)  We use this to prove a sort of quantitative continuity result:
\begin{proposition}\label{prop.inv}
Suppose that $A \in \mathcal{N}(G)$ has $|A^2| \leq K|A|$; $f \in A(G)$ has $\|f\|_{A(G)} \leq M$; and $\epsilon,\eta \in (0,1]$ and $p \geq 2$ are parameters.  Then there are sets $X,B \in \mathcal{N}(G)$ such that $(X,B)$ is an $\eta$-closed pair with $(X^+)^4 \subset A^4$,
\begin{equation*}
|B| \geq \exp(-(\eta^{-1}MK)^{p\exp(O(\epsilon^{-2}))})|A| \text{ and } \sup_t{\|f-f \ast m_X\|_{L_p(m_{tB})}} \leq \epsilon M.
\end{equation*}
\end{proposition}
We prove this result in \S\ref{sec.is}, the key tool is Corollary \ref{cor.keyt} recorded and proved in \S\ref{sec.tools}.

The fact that the argument gives a triply exponential bound is a result of iterative application of Lemma \ref{lem.fr}, but the fact that it is triply-exponential in $O(\epsilon^{-2})$ (rather than, say, $O(\epsilon^{-1})$) comes from the application of the Cotlar-Stein lemma at the end of the proof of Proposition \ref{prop.inv}.  It seems conceivable that this might be improved.

With these tools recorded we are ready to stitch them together to give our main iteration lemma.
\begin{lemma}\label{lem.itlem}
There is an absolute constant $C>0$ such that if $f$ is $\epsilon$-almost integer-valued, $\|f\|_{A(G)} \leq M$, $\eta>0$ is a parameter and $\epsilon \leq \exp(-CM)$, then there is some $H \leq G$ such that $f\ast m_H$ is $(\epsilon + \eta)$-almost integer-valued, $(f\ast m_H)_\Z \not \equiv 0$ and $|H| \geq \exp(-\exp(\exp(O(M^{2}+\log \log \eta^{-1}))))|\supp f_\Z|$.
\end{lemma}
\begin{proof}
Apply Proposition \ref{prop.infstruct} (possible provided $\epsilon \leq \exp(-CM)$) to get $S \subset \supp f_\Z$ such that
\begin{equation*}
|S| = M^{-O(M)}|\supp f_\Z| \text{ and }|SS^{-1}| \leq M^{O(M)}|S|.
\end{equation*}
By Lemma \ref{lem.fr} (applied to $S$ with parameter $1$) there is some $A \in \mathcal{N}(G)$ such that
\begin{equation*}
|A| \geq \exp(-M^{O(1)})|S|, \text{ and }m_{S^{-1}}\ast 1_{SS^{-1}}\ast m_S(x)>\frac{1}{2} \text{ for all }x \in A^4.
\end{equation*}
It follows that
\begin{equation*}
\frac{1}{2}|A^2|\leq \sum_{x \in A^2}{m_{S^{-1}}\ast 1_{SS^{-1}}\ast m_S(x)} \leq |SS^{-1}| \leq M^{O(M)}|S|,
\end{equation*}
so
\begin{equation*}
|A| \geq \exp(-M^{O(1)})|\supp f_\Z| \text{ and }|A^2| \leq \exp(M^{O(1)})|A|.
\end{equation*}
Moreover, if $Z\subset A^4$ then
\begin{equation*}
\frac{1}{2}\leq \langle m_{S^{-1}} \ast 1_{SS^{-1}}\ast m_S,m_Z\rangle = \langle m_S \ast \wt{m_Z},1_{SS^{-1}} \ast m_S\rangle_{\ell_2(G)} \leq \|1_S \ast \wt{m_Z}\|_{\ell_\infty(G)}|SS^{-1}||S|^{-1},
\end{equation*}
so
\begin{equation}\label{eqn.uq}
\sup_t{m_{tZ}(S)}=\|1_S \ast \wt{m_{Z}}\|_{\ell_\infty(G)} \geq M^{-O(M)} \text{ for all }\emptyset \neq Z \subset A^4.
\end{equation}
 Apply Proposition \ref{prop.inv} with parameters $2^{-6}M^{-1}$, $2^{-6}M^{-1}$, and $p$ (the last of which is to be optimised) to get $X,B \in \mathcal{N}(G)$ such that $(X,B)$ is $2^{-6}M^{-1}$-closed, $(X^+)^4 \subset A^4$ and
\begin{equation*}
|B| \geq \exp(-\exp(\exp(O(M^2 +\log p))))|\supp f_\Z| \text{ and }\sup_t{\|f-f\ast m_X\|_{L_p(m_{tB})}} \leq 2^{-6}.
\end{equation*}
We may assume that $\epsilon \leq 2^{-5}$ and hence by the triangle inequality and Lemma \ref{lem.inv} we have that
\begin{equation*}
\|f_\Z - f\ast m_X(t)\|_{L_p(m_{tB})} \leq 2^{-4} \text{ for all }t \in G.
\end{equation*}
It follows that $f\ast m_X$ is $2^{-4}$-almost integer-valued.  By Lemma \ref{lem.inv} for all $t\in G$ and $b \in B^{-1}=B$ we also have
\begin{align*}
|(f \ast m_X)_\Z(tb) - (f \ast m_X)_\Z(t)| & \leq |(f \ast m_X)_\Z(tb)-f \ast m_X(tb)|\\ & \qquad +  |f \ast m_X(tb)-f \ast m_X(t)|\\ & \qquad \qquad + |f \ast m_X(t)-(f \ast m_X)_\Z(t)| < \frac{1}{2}.
\end{align*}
It follows that $(f \ast m_X)_\Z$ is constant on left cosets of $H$, the group generated by $B$.

Now suppose $t \in G$ so that
\begin{align*}
m_{tB}\left(\left\{x \in tB: (f \ast m_X)_\Z(x)\neq f_\Z(x)\right\}\right) &\leq \|(f\ast m_X)_\Z - f_\Z\|_{L_p(m_{tB})}^p\\
& \leq   \big(\|(f\ast m_X)_\Z - f\ast m_X\|_{L_p(m_{tB})}\\
& \qquad \qquad +  \|f\ast m_X - f\|_{L_p(m_{tB})}\\
& \qquad \qquad \qquad \qquad + \|f - f_\Z\|_{L_p(m_{tB})}\big)^p\\
&\leq \left(2^{-4} + 2^{-6}+\epsilon\right)^p\leq 4^{-p},
\end{align*}
and since $H=HB$ it follows that
\begin{align*}
&m_{tH}\left(\left\{x \in tH: (f \ast m_X)_\Z(x)\neq f_\Z(x)\right\}\right)\\
&\qquad \qquad  = \frac{1}{|B|}\sum_{b \in B}{\frac{\left|\left\{x \in tHb: (f \ast m_X)_\Z(x)\neq f_\Z(x)\right\}\right|}{|H|}}\\
 &\qquad \qquad  \qquad \qquad   = \E_{h \in H}{\frac{\left|\left\{x \in thB: (f \ast m_X)_\Z(x)\neq f_\Z(x)\right\}\right|}{|B|}}\\
  &\qquad \qquad  \qquad \qquad  \qquad \qquad   = \E_{h \in H}{m_{thB}\left(\left\{x \in thB: (f \ast m_X)_\Z(x)\neq f_\Z(x)\right\}\right)}\leq 4^{-p}.
\end{align*}
We showed earlier that $(f \ast m_X)_\Z$ is constant on left cosets of $H$ so
\begin{equation*}
(f\ast m_X)_\Z\ast \wt{m_H}(t) = \int{(f\ast m_X)_\Z dm_{tH}} \in \Z.
\end{equation*}
But
\begin{equation*}
|f_\Z\ast \wt{m_H}(t) - (f\ast m_X)_\Z\ast \wt{m_H}(t)| \leq \int{|f_\Z - (f \ast m_X)_\Z|dm_{tH}} =O(M4^{-p}),
\end{equation*}
and so
\begin{equation*}
|f \ast \wt{m_H}(t) - (f\ast m_X)_\Z\ast \wt{m_H}(t)| \leq \epsilon + O(M4^{-p}).
\end{equation*}
Since $H=H^{-1}$ the function $f \ast m_H$ is $(\epsilon + O(M4^{-p}))$-almost integer-valued.

Finally, if $(f\ast m_{H})_\Z \equiv 0$ then $(f\ast m_X)_\Z\equiv 0$ and we have
\begin{equation}\label{eqn.con}
m_{tB}\left(\left\{x \in tB: 0\neq f_\Z(x)\right\}\right) \leq 4^{-p} \text{ for all }t \in G,
\end{equation}
but $B \subset A^4$ and so by (\ref{eqn.uq}) $\sup_t{m_{tB}(\supp f_\Z)} \geq M^{-O(M)}$.  It follows that we may take $p=O(\max\{M\log M,\log M\eta^{-1}\})$ such that $f \ast m_H$ is $(\epsilon+\eta)$-almost integer-valued and we have a contradiction to (\ref{eqn.con}) so that $(f\ast m_{H})_\Z \not \equiv 0$.  The lemma follows.
\end{proof}

\begin{proof}[Proof of Theorem \ref{thm.ky}]
Let $C>0$ be the absolute constant in the statement of Lemma \ref{lem.itlem}.  Let $\epsilon_i:=2^i\epsilon + (2^i-1)2^{-4M^2-6}\exp(-CM)$.  We shall define functions $f_i$ such that
\begin{equation*}
f_i \text{ is $\epsilon_i$-almost integer-valued, } \|f_{i+1}\|_{A(G)}^2 \leq\|f_i\|_{A(G)}^2 - \frac{1}{4},
\end{equation*}
and so that there is a group $H_i \leq G$ and an integer-valued function $z^{(i)} \in \ell_1(G/H_i)$ with
\begin{equation*}
(f_i-f_{i+1})_\Z = \sum_{W \in G/H_i}{z_W^{(i)}1_W} \text{ and }\|z^{(i)}\|_{\ell_1(G/H_i)} \leq \exp(\exp(\exp(O(M^{2})))).
\end{equation*}
We set $f_0:=f$ which is certainly $\epsilon_0$-almost integer-valued (for $\epsilon$ sufficiently small).  At stage $i \leq 4M^2+1$ apply Lemma \ref{lem.itlem} with $\eta=4^{-2M^2-3}\exp(-CM)$ which is possible provided $\epsilon \leq \exp(-C'M^2)$.  We get $H_{i+1} \leq G$ with
\begin{equation*}
|H_{i+1}| \geq \exp(-\exp(\exp(O(M^2))))|\supp (f_i)_\Z| 
\end{equation*}
and
\begin{equation*}
 f_i \ast m_{H_{i+1}} \text{ is }(\epsilon_i+\eta)-\text{almost integer-valued}.
\end{equation*}
Put $f_{i+1}:=f_i - f_i \ast m_{H_{i+1}}$.  Then $f_{i+1}$ is $2\epsilon_i + \eta \leq \epsilon_{i+1}$ almost integer-valued.  For $\epsilon$ sufficiently small we have $\epsilon_i+\eta \leq \frac{1}{4}$ and so if $(f_i \ast m_{H_{i+1}})_\Z(x) \neq 0$ then $|f_i \ast m_{H_{i+1}}(x)| \geq \frac{3}{4}$ by the triangle inequality.  Similarly if $(f_i)_{\Z}(y)=0$ then $|f_i(y)|<\frac{1}{4}$, and so $|f_i(y)| < M\cdot 1_{\supp (f_i)_\Z}(y) + \frac{1}{4}$. Hence
\begin{equation*}
M\frac{|\supp (f_i)_\Z \cap (xH_{i+1})|}{|H_{i+1}|} +\frac{1}{4}=M1_{\supp(f_i)_\Z} \ast m_{H_{i+1}}(x) +\frac{1}{4} > |f_i \ast m_{H_{i+1}}(x)| \geq \frac{3}{4}
\end{equation*}
by the triangle inequality.  Since the left cosets of $H_{i+1}$ partition $G$ and $(f_i\ast m_{H_{i+1}})_\Z$ is invariant on left cosets of $H_{i+1}$ it follows that
\begin{align}\label{eqn.ineq}
|\supp (f_i\ast m_{H_{i+1}})_\Z|&=\sum_{xH_{i+1} \in G/H_{i+1}}{|H_{i+1}|1_{\supp(f_i\ast m_{H_{i+1}})_\Z}(x)}\\ \nonumber
& \leq \sum_{xH_{i+1} \in G/H_{i+1}}{2\cdot M|\supp (f_i)_\Z \cap (xH_{i+1})|}=2M|\supp (f_i)_\Z|.
\end{align}
Again, since the function $(f_i\ast m_{H_{i+1}})_\Z$ is invariant on left cosets of $H_{i+1}$ so it follows from (\ref{eqn.ineq}) and the lower bound on $|H_{i+1}|$ that $(f_i\ast m_{H_{i+1}})_\Z$ takes non-zero integer values on at most $ \exp(\exp(\exp(O(M^{2}))))$ left cosets of $H_{i+1}$.  Added to this, the value of $(f_i\ast m_{H_{i+1}})_\Z$ on each of these is an integer between $-(M+1)$ and $(M+1)$.  It follows that $(f_i-f_{i+1})_\Z=(f_i\ast m_{H_{i+1}})_\Z$ has the claimed form.

Finally, since $(f_i\ast m_{H_{i+1}})_\Z$ is not identically $0$ it follows that $\|f_i\ast m_{H_{i+1}}\|_{A(G)} \geq 1-\epsilon_{i+1} \geq \frac{1}{2}$ and hence $ \|f_{i+1}\|_{A(G)}^2 \leq \|f_i\|_{A(G)}^2 -\|f_i\ast m_{H_{i+1}}\|_{A(G)}^2 \leq\|f_i\|_{A(G)}^2 - \frac{1}{4}$ (by Lemma \ref{lem.split}).  In view of this the iteration terminates in $4M^2$ steps and unpacking what that means we have the result.
\end{proof}

\section{Croot-Sisask lemmas}\label{sec.tools}

The basic tool we need is a slightly adjusted version of \cite[Lemma 3.2]{croabasis::} (the proof of which is the same).  Before beginning we set some standard notation: if $g:\Omega \times Z \rightarrow \C$ and $\omega \in \Omega$ define
\begin{equation*}
g_\omega:Z \rightarrow \C; z \mapsto g(\omega,z).
\end{equation*}
\begin{lemma}\label{lem.cas}
Suppose that $\mu$ is a non-negative measure on a finite set; $\nu$ is a complex-valued measure on a finite set $\Omega$; $p \geq 2$ is a parameter; and $g \in L_p(|\nu| \times \mu)$.  Then there is a function $h$ with $|h(\omega)|=\|\nu\|$ for all $\omega \in \Omega$ and a positive integer $r=O(p\epsilon^{-2})$ such that
\begin{equation*}
|\nu|^r\left(\left\{ \omega \in \Omega^r: \left\| \int{g_{\omega'} d\nu(\omega')} - \frac{1}{r}\left(\sum_{i=1}^r{h(\omega_i)g_{\omega_i}}\right)\right\|_{L_p(\mu)} \leq \epsilon \|\nu\|^{1-\frac{1}{p}} \|g\|_{L_p(|\nu|\times \mu)}\right\}\right) \geq \frac{1}{2}\|\nu\|^r.
\end{equation*}
\end{lemma}
We also need \cite[Proposition 4.2]{san::00} with $S$ replaced by $S^{-1}$ and $T$ replaced by $T^{-1}$.
\begin{lemma}\label{lem.br}
Suppose that $A,S,T \subset G$ are finite and non-empty with $|AS^{-1}| \leq K|A|$ and $|ST| \leq L|S|$, and $k \in \N$ and $\epsilon \in (0,1]$ are parameters.  Then there is some $X \in \mathcal{N}(G)$ such that
\begin{equation*}
|X| \geq \exp(-O(\epsilon^{-2}k^2(\log 2K)( \log 2L)))|T|
\end{equation*}
such that
\begin{equation*}
|m_{A^{-1}} \ast 1_{AS^{-1}}\ast m_S(x)-1| \leq \epsilon \text{ for all }x \in X^k.
\end{equation*}
\end{lemma}
With these recorded we turn to developing the consequences we need.
\begin{lemma*}[Lemma \ref{lem.fr}]
Suppose that $A$ is non-empty and $|AA^{-1}| \leq K|A|$ and $\eta \in (0,1]$.  Then there are $Z,Z^+,Z^-,Y \in \mathcal{N}(G)$ such that $(Z,Y^4;Z^+,Z^-)$ is $\eta$-closed with 
\begin{equation*}
|Y| \geq \exp(-O(\eta^{-2}\log^{O(1)}2K))|A|\text{ and } m_{A^{-1}} \ast 1_{AA^{-1}} \ast m_{A}(x) > \frac{1}{2} \text{ for all }x \in (Z^+)^4.
\end{equation*}
In particular, $(Z^+)^4 \subset A^{-1}AA^{-1}A$.
\end{lemma*}
\begin{proof}
Apply Lemma \ref{lem.br} with $T=A^{-1}$, $S=A$ and $k=48$ to get $W \in \mathcal{N}(G)$ with
\begin{equation*}
|W| \geq \exp(-O(\log^2 2K))|A| \text{ and }|m_{A^{-1}} \ast 1_{AA^{-1}}\ast m_A(x)-1| <\frac{1}{2} \text{ for all }x \in W^{48}.
\end{equation*}
Apply Lemma \ref{lem.br} with all sets equal to $W$ and a parameter $8r$ to get $Y \in \mathcal{N}(G)$ such that
\begin{equation*}
Y^{8r} \subset W^4 \text{ and } |Y| \geq \exp(-O(r^2\log^4K))|A|.
\end{equation*}
Now
\begin{equation*}
\prod_{i=0}^{r-1}{\frac{|Y^8Y^{8i}W^4Y^{8i}Y^8|}{|Y^{8i}W^4Y^{8i}|}} \leq \frac{|W^{12}|}{|W^4|} \leq  \exp(O(\log^2 2K)) ,
\end{equation*}
and so there is some $0 \leq i <r$ such that
\begin{equation*}
|Y^8Y^{8i}W^4Y^{8i}Y^8| \leq \left(\frac{|W^{12}|}{|W^4|}\right)^{\frac{1}{r}}|Y^{8i}W^{4}Y^{8i}|.
\end{equation*}
Let $r=O(\eta^{-1}\log^2K)$ be such that the right hand side is at most $1+\eta$.  Set $Z:=Y^4Y^{8i}W^{4}Y^{8i}Y^4$, $Z^-:=Y^{8i}W^{4}Y^{8i}$ and $Z^+:=Y^8Y^{8i}W^{4}Y^{8i}Y^8$ which are all elements of $\mathcal{N}(G)$ and $Z^-Y^4 \subset Z$ and $Z(Y^4)^{-1} \subset Z^+$ so $(Z,Y^4)$ is $\eta$-closed.  Finally, $(Z^+)^4 \subset W^{48}$ from which the result follows.
\end{proof}
The following lemma is a version of the Croot-Sisask lemma \cite{crosis::} but set up to deal with signedmeasures and to produce covers.  The covers make iterative applications of the lemma easier, essentially because the meet of a cover of size $K$ and a cover of size $L$ is a cover of size at most $KL$.  We shall see this benefit explicitly in the proof of Lemma \ref{lem.int}.
\begin{corollary}\label{cor.keyt}
Suppose that $X,A,S \subset G$ are non-empty with $1_G \in A,S$; $D|A|\geq |AS|$; $\nu$ is a complex-valued measure, absolutely continuous w.r.t. $m_A$, with
\begin{equation*}
\left\|\frac{d\nu}{dm_A}\right\|_{L_2(m_A)} \leq E\left\|\frac{d\nu}{dm_A}\right\|_{L_1(m_A)};
\end{equation*}
and $f \in L_p(m_X)$ for some $p \in [2,\infty)$; and $\epsilon \in (0,1]$ is a parameter.  Then there is a cover $\mathcal{P}$ of $S$ having size at most $\exp(O(\epsilon^{-2}p \log 2DE\epsilon^{-1}))$ such that
\begin{equation*}
\|\rho_{s^{-1}}(f\ast \nu)- \rho_{t^{-1}}(f \ast \nu) \|_{L_p(m_{XAS})} \leq \epsilon \|\nu\| \|f\|_{L_p(m_{XAS})} \text{ for all }s,t \in P \in \mathcal{P}.
\end{equation*}
\end{corollary}
\begin{proof}
We may rescale so that $\|\nu\|=1$. Put $g(x,y):=\rho_{x^{-1}}(f)(y)$ for $|\nu|\times m_{XA}$-a.e.\ $(x,y) \in G^2$ and note that
\begin{equation*}
\|g\|_{L_p(|\nu| \times m_{XA})}^p = \int{\int{|\rho_{x^{-1}}(f)(y)|^pdm_{XA}(y)}d|\nu|(x)}= \|f \|_{L_p(m_{XA})}^p
\end{equation*}
since $\int{g_xd\nu(x)} = f \ast \nu$.  Apply Lemma \ref{lem.cas} with $m_{XA}$ on $G$, the complex-valued measure $\nu$ on $G$, the function $g$, and parameter $\frac{1}{6}\epsilon$ to get a function $h:G \rightarrow \C$ with $|h(x)| =1$ for all $x \in G$ such that the set
\begin{equation*}
\mathcal{L}:=\left\{x \in G^r: \left\| f \ast \nu - \frac{1}{r}\left(\sum_{i=1}^r{h(x_i)\rho_{x_i^{-1}}(f)}\right)\right\|_{L_p(m_{XA})} \leq \frac{1}{6}\epsilon  \|f\|_{L_p(m_{XA})}\right\}
\end{equation*}
has $|\nu|^r(\mathcal{L}) \geq \frac{1}{2}\|\nu\|^r=\frac{1}{2}$.  By the Cauchy-Schwarz inequality we have
\begin{align*}
m_A^r(\mathcal{L})& = \left\|\frac{d\nu}{dm_A}\right\|_{L_2(m_A)}^{-2r}m_A^r(\mathcal{L})\left(\int{\prod_{i=1}^r{\left|\frac{d\nu}{dm_A}(a_i)\right|^{2}}dm_A^r(a)}\right)\\ & \geq \left\|\frac{d\nu}{dm_A}\right\|_{L_2(m_A)}^{-2r}\left|\int{1_\mathcal{L}(a)\prod_{i=1}^r{\left|\frac{d\nu}{dm_A}(a_i)\right|}dm_A^r(a)}\right|^2 =\left\|\frac{d\nu}{dm_A}\right\|_{L_2(m_A)}^{-2r} (|\nu|^r(\mathcal{L}))^2 \geq \frac{1}{4}E^{-2r}.
\end{align*}
We may certainly assume that $\mathcal{L} \subset \supp |\nu|^r \subset A^r$ and so it follows (since $\supp f \ast \nu \subset XA$) that for all $s \in S$ and $x \in \mathcal{L}$ we have
\begin{equation*}
\left\| \rho_{s^{-1}}( f \ast \nu) - \frac{1}{r}\left(\sum_{i=1}^r{h(x_i)\rho_{(x_is)^{-1}}(f)}\right)\right\|_{L_p(m_{XAS})} \leq \frac{1}{6}\epsilon  \|f\|_{L_p(m_{XAS})}.
\end{equation*}
Let $\mathcal{Q}$ be a cover of $(S^1)^r$ of size $\epsilon^{-O(r)}$ such that $\|z-z'\|_{\ell_\infty^r} \leq \frac{1}{6}\epsilon$ for all $z,z' \in Q \in \mathcal{Q}$.  Pulling back this cover along the map $\mathcal{L} \rightarrow (S^1)^r; x \mapsto (h(x_1),\dots,h(x_r))$ it follows that there is some $\mathcal{M} \subset \mathcal{L}$ with
\begin{equation*}
m_A^r(\mathcal{M}) \geq \frac{1}{|\mathcal{Q}|} m_A^r(\mathcal{L}) \text{ s.t. } |h(x_i)-h(x_i')| \leq \frac{1}{6}\epsilon\text{ for all }x,x' \in \mathcal{M}, 1 \leq i \leq r.
\end{equation*}
In particular, if $s,t \in S$ and $x,y\in \mathcal{M}$ are such that $x_is=y_it$ for all $1\leq i \leq r$ then
\begin{align*}
& \left\| \frac{1}{r}\left(\sum_{i=1}^r{h(x_i)\rho_{(x_is)^{-1}}(f)}\right) -  \frac{1}{r}\left(\sum_{i=1}^r{h(y_i)\rho_{(y_it)^{-1}}(f)}\right)\right\|_{L_p(m_{XAS})}\\ & \qquad \qquad \qquad \leq \frac{1}{6}\epsilon \sup{\left\{\left\|\rho_{(as')^{-1}}(f)\right\|_{L_p(m_{XAS})} : s' \in S, a \in A\right\}} \leq \frac{1}{6}\epsilon \|f\|_{L_p(m_{XAS})}
\end{align*}
by the triangle inequality.  Write $\Delta:G \rightarrow G^r; t \mapsto (t,\dots,t)$.  Then
\begin{equation*}
|\mathcal{M}\Delta(S)| \leq |\mathcal{M}(S\times \cdots \times S)| \leq |AS|^r \leq D^r|A|^r \leq 4D^rE^{2r}\epsilon^{-O(r)}|\mathcal{M}|.
\end{equation*}
By Ruzsa's covering lemma (Lemma \ref{lem.rcl}) we see that there is a set $T\subset S$ of size at most $2D^rE^{2r}\epsilon^{-O(r)}$ such that $\{
\mathcal{M}^{-1}\mathcal{M}\Delta(t) : t \in T\}$ covers $\Delta(S)$; let
\begin{equation*}
\mathcal{P}:=\{\{s: \Delta(s) \in \mathcal{M}^{-1}\mathcal{M}\Delta(t)\}: t \in T\}
\end{equation*}
so that $\mathcal{P}$ is a cover of $S$ of the claimed size.  Finally, if $s_0,s_1 \in P \in \mathcal{P}$ then there is some $t \in T$ and elements $x^{(0)},x^{(1)},y^{(0)},y^{(1)} \in \mathcal{M}$ such that
\begin{equation*}
x^{(0)}_is_0 = y^{(0)}_it \text{ and } x^{(1)}_is_1 = y^{(1)}_it \text{ for all }1 \leq i \leq r.
\end{equation*}
For $j \in \{0,1\}$ we then have
\begin{align*}
\left\|\rho_{s_j^{-1}}( f \ast \nu) - \rho_{t^{-1}}( f \ast \nu)\right\|_{L_p(m_{XAS})}  \leq &\left\|\rho_{s_j^{-1}}( f \ast \nu) -  \frac{1}{r}\left(\sum_{i=1}^r{h(x_i^{(j)})\rho_{(x_i^{(j)}s_j)^{-1}}(f)}\right)\right\|_{L_p(m_{XAS})}\\
&+ \left\| \frac{1}{r}\left(\sum_{i=1}^r{h(x_i^{(j)})\rho_{(x_i^{(j)}s_j)^{-1}}(f)}\right)\right.\\
&\qquad\qquad\left.- \frac{1}{r}\left(\sum_{i=1}^r{h(y_i^{(j)})\rho_{(y_i^{(j)}t)^{-1}}(f)}\right)\right\|_{L_p(m_{XAS})}\\
&  + \left\| \frac{1}{r}\left(\sum_{i=1}^r{h(y_i^{(j)})\rho_{(y_i^{(j)}t)^{-1}}(f)}\right)- \rho_{t^{-1}}( f \ast \nu)\right\|_{L_p(m_{XAS})}.
\end{align*}
It follows that
\begin{equation*}
\left\|\rho_{s_j^{-1}}( f \ast \nu) - \rho_{t^{-1}}( f \ast \nu)\right\|_{L_p(m_{XAS})} \leq \frac{1}{2}\epsilon\|f\|_{L_p(m_{XAS})},
\end{equation*}
and we get the conclusion by the triangle inequality again.
\end{proof}
We do \emph{not} actually need the $p$-dependence in the above result.  Although Proposition \ref{prop.inv} does use good $L_p$-control of the norm it achieves this by pigeonholing the $A(G)$-mass on the spectral side since $L_p$ is dominated by $L_\infty$ which, in turn, is dominated by the $A(G)$ norm.  Nevertheless the proof for general $p$ is not significantly more involved and may be of use elsewhere.

\section{Invariant sets}\label{sec.is}

In this section we shall prove the following proposition.  The argument is not dissimilar to the basic scheme in \cite{grekon::} which, itself, is a sort of regularity argument.  
\begin{proposition*}[Proposition \ref{prop.inv}]
Suppose that $A \in \mathcal{N}(G)$ has $|A^2| \leq K|A|$; $f \in A(G)$ has $\|f\|_{A(G)} \leq M$; and $\epsilon,\eta \in (0,1]$ and $p \geq 2$ are parameters.  Then there are sets $X,X^+,X^-,B \in \mathcal{N}(G)$ such that $(X,B;X^+,X^-)$ is an $\eta$-closed pair with $(X^+)^4 \subset A^4$,
\begin{equation*}
|B| \geq \exp(-(\eta^{-1}MK)^{p\exp(O(\epsilon^{-2}))})|A| \text{ and } \sup_t{\|f-f \ast m_X\|_{L_p(m_{tB})}} \leq \epsilon M.
\end{equation*}
\end{proposition*}
The proof is iterative and based around the following lemma.
\begin{lemma}\label{lem.int}
Suppose that $f \in A(G)$ has $\|f\|_{A(G)} = M$; $\eta \in (0,\frac12]$ is a parameter; $B \in \mathcal{N}(G)$ has $(B,T;B^+,B^-)$ $\eta$-closed; $X \subset G$ has $K|X| \geq |XB^+|$; and $\epsilon \in (0,\frac1]$ is a further parameter.  Then there is a set $W \subset T$ such that
\begin{equation*}
|W| \geq \exp(-(2\epsilon^{-1} MK)^{O(1)})|T|
\end{equation*}
and
\begin{equation*}
\|\rho_{u^{-1}}(f)-\rho_{t^{-1}}(f) \|_{L_2(m_{B^-})} \leq \epsilon +\eta^{\frac12} (MK)^{O(1)} \text{ for all }u,t\in W.
\end{equation*}
\end{lemma}
\begin{proof}
Since $f \in A(G)$, Lemma \ref{lem.various} tells us that there is a constant (which we may as well take to be $M$); a probability space $(\Omega,\P)$; and functions $h_\omega,g_\omega \in L_2(m_G)$ with
\begin{equation*}
\|h_\omega\|_{L_2(m_G)}\leq 1 \text{ and } \|g_{\omega}\|_{L_2(m_G)} \leq 1 \text{ for all } \omega \in \Omega,
\end{equation*}
such that
\begin{equation*}
f(x)=M\E_{\omega}{\wt{h_\omega} \ast g_\omega(x)} \text{ for all }x \in G.
\end{equation*}
For $z \in G$ we have
\begin{equation*}
\int{\frac{\wt{1_{sX}}(z)}{m_G(sX)}dm_G(s)}=\int{\frac{1_{z^{-1}X^{-1}}(s)}{m_G(X)}dm_G(s)}=1
\end{equation*}
and so if $x \in B$ then
\begin{align*}
f(x) & = M\E_{\omega}{\int{\wt{h_\omega}(xy^{-1})g_\omega(y)\frac{\wt{1_{sX}}(xy^{-1})}{m_G(sX)}dm_G(s)dm_G(y)}}\\
& =  M\E_{\omega}{\int{(h_\omega dm_{sX})^\sim(xy^{-1})(g_\omega 1_{sXB})(y)dm_G(y)dm_G(s)}}.
\end{align*}
For each $y \in G$ we put
\begin{equation*}
h_{\omega,s}(y):=\frac{1}{\|h_\omega\|_{L_2(m_{sX})}}h_\omega(y)1_{sX}(y) \text{ and } g_{\omega,s}(y):=\frac{1}{\|g_\omega\|_{L_2(m_{sXB})}}g_\omega(y)1_{sXB}(y),
\end{equation*}
and with this notation we see that (for $x \in B$)
\begin{align*}
f(x) & =  M\E_{\omega,s}{\|h_\omega\|_{L_2(m_{sX})}\|g_\omega\|_{L_2(m_{sXB})}  (h_{\omega,s}dm_{sX})^\sim \ast g_{\omega,s}(x)}.
\end{align*}
By the Cauchy-Schwarz inequality we have (for each $\omega \in \Omega$)
\begin{align*}
\E_{s}{\|h_\omega\|_{L_2(m_{sX})}\|g_\omega\|_{L_2(m_{sXB})}} &\leq \left(\E_{s}{\|h_\omega\|_{L_2(m_{sX})}^2}\right)^{1/2}\left(\E_{s}{\|g_\omega\|_{L_2(m_{sXB})}^2}\right)^{1/2}\leq 1.
\end{align*}
It follows that
\begin{equation*}
\E_{\omega,s}{\|h_\omega\|_{L_2(m_{sX})}\|g_\omega\|_{L_2(m_{sXB})}} \leq 1,
\end{equation*}
and hence there is a probability measure $\P'$ on $\Omega \times G$ and a constant $M' \leq M$ such that
\begin{equation*}
f(x)=M'\E'{ (h_{\omega,s}dm_{sX})^\sim\ast g_{\omega,s}(x)} \text{ for all }x \in B.
\end{equation*}
Note that
\begin{align*}
 \|(h_{\omega,s}dm_{sX})^\sim \ast g_{\omega,s}\|_{L_\infty(G)}  \leq \sqrt{\frac{m_G(sXB)}{m_G(sX)}}\|h_{\omega,s}\|_{L_2(m_{sX})}\|g_{\omega,s}\|_{L_2(m_{sXB})} \leq \sqrt{K},
\end{align*}
and so
\begin{align*}
\left(\E'{\|M'(h_{\omega,s}dm_{sX})^\sim \ast g_{\omega,s}\|_{L_2(m_{B})}^2}\right)^{\frac{1}{2}}  \leq M'\sqrt{K}.
\end{align*}
It follows by Lemma \ref{lem.cas} (applied to $((\omega,s),x) \mapsto M'(h_{\omega,s}dm_{sX})^\sim \ast g_{\omega,s}(x)$ with $p=2$, $\mu=m_{B}$, $\nu=\P'$ and with some parameter $\delta$ chosen later) that there is some $r=O(\delta^{-2})$, $\omega_1,\dots,\omega_r \in \Omega$ and $s_1,\dots,s_r \in G$ such that
\begin{equation*}
\left\| f - \frac{M'}{r}\left(\sum_{i=1}^r{(h_{\omega_i,s_i}dm_{s_iX})^\sim \ast g_{\omega_i,s_i} }\right)\right\|_{L_2(m_{B})} \leq \delta \sqrt{K}M'.
\end{equation*}
Write $\gamma_i:=(h_{\omega_i,s_i}dm_{s_iX})^\sim \ast g_{\omega_i,s_i}$ so that $\|\gamma_i\|_{L_\infty(G)} \leq \sqrt{K}$ and
\begin{equation}\label{eqn.inserty}
\left\| \rho_{t^{-1}}(f) - \frac{M'}{r}\left(\sum_{i=1}^r{\rho_{t^{-1}}(\gamma_i) }\right)\right\|_{L_2(m_{B^-})} \leq 2\delta \sqrt{K}M' \text{ for all }t \in T.
\end{equation}
For each $1\leq i\leq r$ set $\nu:=\gamma_idm_B$, and if $\|\gamma_i\|_{L_1(m_B)} \leq \frac{1}{2}\delta\sqrt{K}$ then for all $t,u \in G$ we have
\begin{equation*}
\|\rho_{t^{-1}}(h_{\omega_i,s_i}\ast (\gamma_idm_{B})) - \rho_{u^{-1}}(h_{\omega_i,s_i}\ast ( \gamma_idm_{B}))\|_{L_2(m_{s_iXBT})} \leq \delta \sqrt{K}\| h_{\omega_i,s_i}\|_{L_2(m_{s_iXBT})}.
\end{equation*}
On the other hand if $\|\gamma_i\|_{L_1(m_B)} \geq \frac{1}{2}\delta\sqrt{K}$, then
\begin{equation*}
\left\|\frac{d\nu}{dm_B}\right\|_{L_2(m_B)}\leq \frac{\|\gamma_i\|_{L_\infty(m_B)}}{\|\gamma_i\|_{L_1(m_B)}}{\left\|\frac{d\nu}{dm_B}\right\|_{L_1(m_B)}}  \leq 2\delta^{-1}{\left\|\frac{d\nu}{dm_B}\right\|_{L_1(m_B)}} .
\end{equation*}
Apply Corollary \ref{cor.keyt} with sets $s_iX$, $B$ and $T$, measure $\nu$, function $h_{\omega_i,s_i} \in L_2(m_{s_iX})$, and parameters $2$ and $\delta$ to get a cover $\mathcal{P}_{i}$ of $T$ of size $\exp(O(\delta^{-4}\log 2\delta^{-1}))$ such that for all $t,u \in P \in \mathcal{P}_{i}$ we have
\begin{align*}
&\|\rho_{t^{-1}}(h_{\omega_i,s_i}\ast (\gamma_idm_{B})) - \rho_{u^{-1}}(h_{\omega_i,s_i}\ast ( \gamma_idm_{B}))\|_{L_2(m_{s_iXBT})}\\
&\qquad \qquad \qquad  \qquad \qquad \leq \delta\|\gamma_i\|_{L_1(m_{B})}\| h_{\omega_i,s_i}\|_{L_2(m_{s_iXBT})} \leq \delta \sqrt{K}\| h_{\omega_i,s_i}\|_{L_2(m_{s_iXBT})}.
\end{align*}
Let $\mathcal{P}:=\bigwedge_{i=1}^r{\mathcal{P}_{i}}$ which is a cover of $T$ of size at most $\exp(O(r\delta^{-4}\log 2\delta^{-1}))$, and suppose that $W$ is the largest cell so that we have the size bound for $W$ we require (once we have chosen $\delta$), and if $t,u \in W$ and $x \in T$ then
\begin{align*}
& |\langle \rho_{t^{-1}}(h_{\omega_i,s_i}\ast (\gamma_idm_{B}))- \rho_{u^{-1}}(h_{\omega_i,s_i}\ast ( \gamma_idm_{B})), \rho_{x^{-1}}(g_{\omega_i,s_i})\rangle_{L_2(m_{s_iXBT})}|\\ &
\qquad \qquad \qquad \qquad \qquad  \leq \delta\sqrt{K}\|\rho_{x^{-1}}(g_{\omega_i,s_i})\|_{L_2(m_{s_iXBT})}\| h_{\omega_i,s_i}\|_{L_2(m_{s_iXBT})}\\
& \qquad \qquad \qquad \qquad \qquad \qquad \qquad \qquad = \delta \sqrt{K}\|g_{\omega_i,s_i}\|_{L_2(m_{s_iXBT})}\| h_{\omega_i,s_i}\|_{L_2(m_{s_iXBT})} .
\end{align*}
On the other hand
\begin{align*}
&\langle \rho_{t^{-1}}(h_{\omega_i,s_i}\ast (\gamma_idm_{B}))- \rho_{u^{-1}}(h_{\omega_i,s_i}\ast ( \gamma_idm_{B})), \rho_{x^{-1}}(g_{\omega_i,s_i})\rangle_{L_2(m_{s_iXBT})}\\ &
 \qquad \qquad =\frac{|X|}{|XBT|}\langle (h_{\omega_i,s_i}dm_{s_iX})\ast (\rho_{t^{-1}}(\gamma_idm_{B})- \rho_{u^{-1}}( \gamma_idm_{B})), \rho_{x^{-1}}(g_{\omega_i,s_i})\rangle\\
 &
 \qquad \qquad  \qquad \qquad =\frac{|X|}{|XBT|}\langle \rho_{t^{-1}}(\gamma_idm_{B})- \rho_{u^{-1}}( \gamma_idm_{B}), \rho_{x^{-1}}((h_{\omega_i,s_i}dm_{s_iX})^\sim\ast g_{\omega_i,s_i})\rangle\\
 & \qquad \qquad  \qquad \qquad \qquad \qquad =\frac{|X|}{|XBT|}\frac{|BT|}{|B|}\langle \rho_{t^{-1}}(\gamma_i|_B)- \rho_{u^{-1}}( \gamma_i|_B), \rho_{x^{-1}}(\gamma_i)\rangle_{L_2(m_{BT})}.
\end{align*}
Combining the above we get
\begin{equation*}
|\langle \rho_{t^{-1}}(\gamma_i|_B) - \rho_{u^{-1}}(\gamma_i|_B),\rho_{x^{-1}}(\gamma_i)\rangle_{L_2(m_{BT})}|  \leq \delta K^{3/2}.
\end{equation*}
Since $(B,T)$ is $\eta$-closed we have
\begin{equation*}
\frac{|BT|}{|B|}\cdot \|\rho_{x^{-1}}(\gamma_i)-\rho_{x^{-1}}(\gamma_i|_B)\|_{L_1(m_{BT})} \leq 2\|\gamma_i\|_{L_\infty(G)}\frac{|B^+ \setminus B^-|}{|B|} = O(\eta \sqrt{K})
\end{equation*}
and it follows that
\begin{equation*}
|\langle \rho_{t^{-1}}(\gamma_i|_B) - \rho_{u^{-1}}(\gamma_i|_B),\rho_{x^{-1}}(\gamma_i|_B)\rangle_{L_2(m_{BT})}|  \leq \delta K^{3/2} +O(\eta K).
\end{equation*}
Since $x \in T$ was arbitrary we conclude from the cases $x=u$ and $x=t$ that
\begin{align*}
\sqrt{\frac{|B^-|}{|B^+|}}\left\|\rho_{t^{-1}}(\gamma_i) - \rho_{u^{-1}}(\gamma_i)\right\|_{L_2(m_{B^-})}& \leq \left\|\rho_{t^{-1}}(\gamma_i|_B) - \rho_{u^{-1}}(\gamma_i|_B)\right\|_{L_2(m_{BT})}  \leq O(\delta K^{3/2})+O(\eta K)
\end{align*}
for all $t,u \in W$.  By the triangle inequality we then have
\begin{equation*}
\left\|\frac{M'}{r}\left(\sum_{i=1}^r{\rho_{t^{-1}}(\gamma_i) }\right)-\frac{M'}{r}\left(\sum_{i=1}^r{\rho_{u^{-1}}(\gamma_i) }\right)\right\|_{L_2(m_{B^-})} =O(M\delta K^{3/2}+ \eta M K),
\end{equation*}
for all $t,u \in W$, and combining this with (\ref{eqn.inserty}) taking $\delta = \Omega(\epsilon K^{-3/2} M^{-1})$ gives the result.
\end{proof}

\begin{proof}[Proof of Proposition \ref{prop.inv}]
Let $H$ be a Hilbert space, $\pi:G \rightarrow \Aut(H)$ a homomorphism and $v,w\in H$ have $\|w\|=1$ and $\|v\| \leq M$ such that $f(x)=\langle\pi(x)v,w\rangle$ for all $x \in G$. (This is possible by Lemma \ref{lem.various}.)

We begin by fixing some auxiliary parameters.  Let $\nu$ and $n$ be parameters (we will have $n=O(\epsilon^{-2})$ and $\nu=\Omega(\epsilon^{4})$) that will be optimised by later choices. Let $\delta = \epsilon^{O(p)} M^{-O(1)}$ be such that if $\eta \leq \delta$ in Lemma \ref{lem.int} then the $\eta (MK)^{O(1)}$ error term is at most $\frac{1}{4^p}\epsilon^{p/2}$ if $K\leq 2$.

We construct $X_i,Z_{i+1},B_i,T_i \in \mathcal{N}(G)$ for $0 \leq i \leq n-1$ such that
\begin{equation*}
X_i \subset A_0:=\left\{x:m_A \ast 1_{A^2}\ast m_A(x)>\frac{1}{2}\right\};
\end{equation*}
$(X_i,Z_{i+1})$ and $(Z_{i+1},X_{i+1})$ are $\nu$-closed; $(B_i,T_i^4)$ is $\delta$-closed; $(X_i,B_i^4)$ is $\min\{\eta,\nu\}$-closed; and $|A_0| \leq K_i|T_i|$.  (We shall calculate how $K_i$ evolves later as it is a little more complicated.)  Finally, we shall ensure
\begin{equation}\label{eqn.u}
\left\| \left(\wh{m_{Z_{i+1}}}(\pi)-\wh{m_{X_{i}}}(\pi)\right)v\right\|>\frac{1}{2}\epsilon M.
\end{equation}
Apply Lemma \ref{lem.fr} to $A$ with parameter $\min\{\eta,\nu\}$ to get sets $X_0,Y_0 \in \mathcal{N}(G)$ such that $X_0 \subset A_0$; $(X_0,Y_0^4)$ is $\min\{\eta,\nu\}$-closed; and
\begin{equation*}
|Y_0| \geq \exp(-O(\eta^{-2}\nu^{-2}\log^{O(1)}K))|A|.
\end{equation*}
Now $|A_0| \leq 2K|A|$, and since $Y_0^4 \subset X_0 \subset A_0$ we conclude that
\begin{equation*}
|Y_0Y_0^{-1}| = |Y_0^2| \leq |A_0| \leq  \exp(O(\eta^{-2}\nu^{-2}\log^{O(1)}K))|Y_0|.
\end{equation*}
Apply Lemma \ref{lem.fr} again to $Y_0$ with parameter $\delta$ to get $B_0,T_0\in \mathcal{N}(G)$ such that $B_0^4 \subset Y_0^4$ (so $(X_0,B_0^4)$ is $\min\{\eta,\nu\}$-closed); $(B_0,T_0^4)$ is $\delta$-closed; and
\begin{equation*}
|T_0| \geq \exp(-(\delta^{-1}\eta^{-1}\nu^{-1}\log K)^{O(1)})|Y_0| \geq \exp(-(\delta^{-1}\eta^{-1}\nu^{-1}\log K)^{O(1)})|A_0|,
\end{equation*}
which gives our bound on $K_0$.

Suppose we are at step $i \leq n-1$ (so that $X_i$, $B_i$ and $T_i$, but \emph{not} $Z_{i+1}$, have been defined) and there is some $t \in G$ such that
\begin{equation}\label{eqn.boost}
\|f -f\ast m_{X_i}\|_{L_p(m_{tB_i^-})} > \epsilon M.
\end{equation}
Since $f(tb) = \langle \pi(t)\pi(b)v,w\rangle = \langle \pi(b)v,\pi(t)^*w\rangle$ we can replace $w$ by $\pi(t)^*w$ and assume $t=1_G$.  Apply Lemma \ref{lem.int} with sets $X_i$, and $(B_i,T_i;B_i^-T_i^3,B_iT_i)$ $\delta$-closed to get a set $W_{i} \subset T_i$ with
\begin{equation*}
|W_i| \geq \exp(-(2^p\epsilon^{-p}M)^{O(1)})|T_i|
\end{equation*}
such that
\begin{equation*}
\|\rho_{u^{-1}}(f) - \rho_{t^{-1}}(f)\|_{L_2(m_{B_i^-T_i^3})} \leq \frac{2}{4^p}\epsilon^{p/2} \text{ for all }u,t \in W_i,
\end{equation*}
from the choice of $\delta$.  It follows that for all $r,s,t,u \in W_i$ we have
\begin{align}
\nonumber \|f - \rho_{rs^{-1}ut^{-1}}(f)\|_{L_2(m_{B_i^{-}})}& \leq \|f - \rho_{rs^{-1}}(f)\|_{L_2\left(m_{B_i^{-}}\right)}+\|\rho_{rs^{-1}}(f) - \rho_{rs^{-1}ut^{-1}}(f)\|_{L_2\left(m_{B_i^{-}}\right)}\\
\nonumber & \leq \sqrt{\frac{|B_i^-T_i|}{|B_i^-|}}\|\rho_{r^{-1}}(f) - \rho_{s^{-1}}(f)\|_{L_2\left(m_{B_i^-T_i}\right)}\\
\nonumber &\qquad \qquad \qquad  + \sqrt{\frac{|B_i^-T_i^2|}{|B_i^-|}}\|f - \rho_{ut^{-1}}(f)\|_{L_2\left(m_{B_i^-T_i^2}\right)}\\
\nonumber & \leq \sqrt{\frac{|B_i^-T_i^3|}{|B_i^-|}}\|\rho_{r^{-1}}(f) - \rho_{s^{-1}}(f)\|_{L_2\left(m_{B_i^-T_i^3}\right)}\\
\nonumber & \qquad \qquad \qquad +  \sqrt{\frac{|B_i^-T_i^3|}{|B_{i}^-T_i^2|}}\|\rho_{u^{-1}}(f) - \rho_{t^{-1}}(f)\|_{L_2\left(m_{B_i^-T_i^3}\right)}\\
\label{eqn.upit}& \leq  2^{3-2p}\epsilon^{p/2}.
\end{align}
Since $W_i^{-1} \subset T_i^{-1}=T_i$ and $T_i^2\subset B_i \subset X_i \subset A_0$ we see that
\begin{equation*}
|W_i^{-1}W_i| \leq |A_0| \leq K_i|T_i| \leq \exp((2^p\epsilon^{-p}M)^{O(1)})K_i|W_i|.
\end{equation*} 
Apply Lemma \ref{lem.fr} to $W_i^{-1}$ with parameter $\nu$ to get $Z_{i+1}\in \mathcal{N}(G)$ such that $Z_{i+1}^4 \subset W_iW_i^{-1}W_iW_i^{-1}\subset T_i^4$; $(Z_{i+1},Y_{i}^4)$ is $\nu$-closed; and
\begin{equation*}
|Y_i| \geq \exp(-(\nu^{-1}2^p\epsilon^{-p}M\log K_i)^{O(1)})|W_i|.
\end{equation*}
In particular, since $Y_i^2\subset Z_{i+1}^2\subset T_i^4 \subset B_i\subset X_i \subset A_0$ we have
\begin{equation}\label{eqn.ydoub}
|Y_i^2| \leq |A_0| \leq K_i|T_i| \leq  \exp((\nu^{-1}2^p\epsilon^{-p}M\log K_i)^{O(1)})|Y_i|.
\end{equation} 
It follows from the triangle inequality applied to (\ref{eqn.upit}) that
\begin{equation*}
\|f - f\ast m_{Z_{i+1}}\|_{L_p(m_{B_i^{-}})}^p \leq (2M)^{p-2}\|f - f\ast m_{Z_{i+1}}\|_{L_2(m_{B_i^{-}})}^2 \leq \left(\frac{1}{2}\epsilon M\right)^p.
\end{equation*}
By the triangle inequality again and (\ref{eqn.boost}) (recalling that we have argued we may take $t=1_G$) we have
\begin{equation*}
\| f\ast m_{Z_{i+1}}-f \ast m_{X_i}\|_{A(G)}\geq \| f\ast m_{Z_{i+1}}-f \ast m_{X_i}\|_{L_p(m_{B_i^-})} >\frac{1}{2}\epsilon M,
\end{equation*}
and it follows from this that (\ref{eqn.u}) holds.

Apply Lemma \ref{lem.fr} to $Y_i$ (using the bound in (\ref{eqn.ydoub})) with parameter $\min\{\nu,\eta\}$ to get $X_{i+1},U_{i+1} \in \mathcal{N}(G)$, such that $(X_{i+1},U_{i+1}^4)$ is $\min\{\nu,\eta\}$-closed; $(X_{i+1}^+)^4 \subset Y_{i}^4$; and 
\begin{equation*}
|U_{i+1}| \geq \exp(-(\eta^{-1}\nu^{-1}2^p\epsilon^{-p}M\log K_i)^{O(1)})|Y_i|.
\end{equation*}
Moreover, $X_{i+1}\subset Y_i^4 \subset Z_{i+1}\subset T_i^4\subset B_i\subset X_i \subset A_0$ and $U_{i+1}^2 \subset X_{i+1}$ so
\begin{equation*}
|U_{i+1}^2| \leq |A_0| \leq \exp((\eta^{-1}\nu^{-1}2^p\epsilon^{-p}M\log K_i)^{O(1)})|U_{i+1}|.
\end{equation*}
Finally, apply Lemma \ref{lem.fr} to $U_{i+1}$ with parameter $\delta$ to get $B_{i+1},T_{i+1}\in \mathcal{N}(G)$ such that $(B_{i+1},T_{i+1}^4)$ is $\delta$-closed; $(B_{i+1}^+)^4 \subset U_{i+1}^4$ and
\begin{equation*}
|T_{i+1}| \geq \exp(-(\delta^{-1}\eta^{-1}\nu^{-1}2^p\epsilon^{-p}M\log K_i)^{O(1)})|U_{i+1}|.
\end{equation*}
Since $B_{i+1}^4 \subset U_{i+1}^4$ and $(X_{i+1},U_{i+1}^4)$ is $\min\{\nu,\eta\}$-closed we conclude that $(X_{i+1},B_{i+1}^4)$ is $\min\{\nu,\eta\}$-closed.  Since $Z_{i+1}^4 \subset T_i^4\subset B_i^4$ we conclude that $(X_i,Z_{i+1})$ is $\nu$-closed, and since $X_{i+1}\subset Y_i^4$ and $(Z_{i+1},Y_i^4)$ is $\nu$-closed we conclude that $(Z_{i+1},X_{i+1})$ is $\nu$-closed as required.

We repeat this provided $i \leq n-1$ and there is some $t \in G$ such that (\ref{eqn.boost}) happens, and we shall see for a suitable choice of $n$ that that leads to a contradiction.

Since $(X_i,Z_{i+1})$ and $(Z_{i+1},X_{i+1})$ are $\nu$-closed for $0 \leq i \leq n-1$, if $j>i$ then $Z_j \subset X_j \subset X_{i+1}$ and so
\begin{equation*}
\|m_{Z_{i+1}} \ast m_{X_j} - m_{Z_{i+1}}\|,\|m_{Z_{i+1}} \ast m_{Z_{j+1}} - m_{Z_{i+1}}\| \leq \nu,
\end{equation*}
and 
\begin{equation*}
\|m_{X_i} \ast m_{Z_{j+1}} - m_{X_i}\|,\|m_{X_i} \ast m_{X_j} - m_{X_i}\| \leq \nu.
\end{equation*}
Dealing with $i>j$ similarly we have
\begin{equation*}
\|\left(\wh{m_{Z_{i+1}}}(\pi)-\wh{m_{X_{i}}}(\pi)\right)\left(\wh{m_{Z_{j+1}}}(\pi)-\wh{m_{X_{j}}}(\pi)\right)\| \leq \begin{cases} 4 & \text{ if }i=j\\ 4\nu & \text{ if } i \neq j\end{cases}.
\end{equation*}
It follows from the (finite version of the) Cotlar-Stein lemma that for any signs $\sigma_0,\dots,\sigma_{n-1} \in \{-1,1\}$ we have
\begin{equation*}
\left\|\sum_{i=0}^{n-1}{\sigma_i\left(\wh{m_{Z_{i+1}}}(\pi)-\wh{m_{X_{i}}}(\pi)\right)}\right\| \leq O(1+\sqrt{\nu} n).
\end{equation*}
We conclude (from (\ref{eqn.u})) that
\begin{align*}
n\left(\frac{1}{2}\epsilon M\right)^2 &\leq \sum_{i=0}^{n-1}{\left\|\left(\wh{m_{Z_{i+1}}}(\pi)-\wh{m_{X_{i}}}(\pi)\right)v\right\|^2}\\
& = \E_\sigma{\left\|\sum_{i=0}^{n-1}{\sigma_i\left(\wh{m_{Z_{i+1}}}(\pi)-\wh{m_{X_{i}}}(\pi)\right)v}\right\|^2} \leq O(1+\sqrt{\nu} n)^2\|v\|^2,
\end{align*}
where the $\sigma_0,\dots,\sigma_{n-1} \in \{-1,1\}$ are chosen independently and uniformly.  It follows that there is some $\nu = \Omega(\epsilon^4)$ and $n=O(\epsilon^{-2})$ leading to a contradiction as claimed.

It remains to note
\begin{equation*}
\log K_{i+1} \leq (\eta^{-1} \delta^{-1}\nu^{-1}2^p\epsilon^{-p}M\log K_i)^{O(1)}.
\end{equation*}
Since the iteration proceeds at most $i_0=O(\epsilon^{-2})$ times this gives us a bound on $B_{i_0}$ when the iteration terminates (since $|B_{i_0}| \geq |T_{i_0}|$).  We set $X:=X_{i_0}$ and $B:=B_{i_0}$ which have the required properties.
\end{proof}

\section{Arithmetic connectivity}\label{sec.ac}

If $f$ is Boolean and has $\|f\|_{A(G)} \leq M$ then $\log$-convexity of $L_p$-norms can be used to show that $E(f) \geq M^{-2}|\supp f|^3$.  However, when $f$ is only almost integer-valued then it may be that most of the mass of $E(f)$ may be supported outside $\supp f_\Z$.  To deal with this in the abelian setting Green introduced the concept of arithmetic connectivity in \cite[Definition 5.2]{gresan::}.  In this section we develop this in the setting of general groups and prove the following proposition.
\begin{proposition*}[Proposition \ref{prop.infstruct}]
There is an absolute $C>0$ such that if $f$ is $\epsilon$-almost integer-valued with $\|f\|_{A(G)} \leq M$ and $\epsilon \leq \exp(-CM)$, then there is some $S \subset \supp f_\Z$ such that $|SS^{-1}| \leq M^{O(M)}|S|$ and $|S| = M^{-O(M)}|\supp f_\Z|$.
\end{proposition*}
Given $k,t \in \N$ we are interested in the vectors in $[k]^{t}=\{1,\dots,k\}^t$.  We say that $i \in [k]^{t}$ is \textbf{trivial} if
\begin{equation*}
|\{j \in [k]: |\{s \in [t]: i_s=j\}|=1\}| \leq 1;
\end{equation*}
we write $T_{k,t}$ for the set of trivial elements of $[k]^{t}$, and $N_{k,t}$ for the set of non-trivial elements, that is $N_{k,t}:=[k]^{t}\setminus T_{k,t}$.  In words we say that $i \in [k]^t$ is trivial if at most one coordinate is unique.

Given a group $G$ we say that $A \subset G$ is \textbf{$(k,l)$-arithmetically connected} if for every $x \in A^k$ there is some $1 \leq r \leq l$, some $i \in N_{k,2r+1}$, and some $\sigma \in \{-1,1\}^{2r+1}$ such that
\begin{equation*}
x_{i_1}^{\sigma_1}\cdots x_{i_{2r+1}}^{\sigma_{2r+1}} \in A.
\end{equation*}

It is useful to know roughly how many trivial tuples there are and we capture this in the following lemma.
\begin{lemma}\label{lem.ct}
Given $k,r \in \N$ we have
\begin{equation*}
|T_{k,2r+1}| \leq \exp(O(r))r^{r}k^{r+1}.
\end{equation*}
\end{lemma}
\begin{proof}
Writing
\begin{equation*}
R_{k,r}:=\{i:[2r] \rightarrow [k] \text{ s.t. } |i^{-1}(\{j\})|\neq 1 \text{ for all }j \in [k]\},
\end{equation*}
we see that there is a natural surjection
\begin{equation*}
[2r+1] \times [k] \times R_{k,r} \rightarrow T_{k,2r+1}
\end{equation*}
so that $|T_{k,2r+1}| \leq (2r+1)k|R_{k,r}|$.  The result will follow from an upper bound on this last quantity.

Now $|R_{k,r}|$ is the coefficient of $x^{2r}$ in $(2r)!(\exp(x)-x)^k$. The coefficients of $x^n$ in this expression are non-negative and so for any $\eta >0$ we have
\begin{equation*}
|R_{k,r}| \leq (2r)!\eta^{-2r}(\exp(\eta)-\eta)^k =(2r)!\eta^{-2r}\exp(O(\eta^2k)).
\end{equation*}
Optimising with $\eta^2k =r$ we end up with $|R_{k,r}| \leq \exp(O(r))r^rk^r$ from which the result follows.
\end{proof}

The proof of the next result is inspired by \cite[Lemma 1]{mel::} of M{\'e}la and uses Chebychev polynomials, the basic properties of which may be found in \cite[\S6.10.6]{zwikraros::}.

Keeping the second parameter small is most important with the below but if the first parameter is a concern then it may be useful to know that the auxiliary measures in \cite[Lemme 4]{mel::} can be used to show that $\supp f_\Z$ is $(O(M^2\log M),O(M\log M))$-arithmetically connected.
\begin{lemma}\label{lem.st}
There is an absolute $C>0$ such that if $f$ is $\epsilon$-almost integer-valued with $\|f\|_{A(G)} \leq M$ and $\epsilon \leq \exp(-CM)$, then $\supp f_\Z$ is $(O(M^3),O(M))$-arithmetically connected.
\end{lemma}
\begin{proof}
First replace $f$ by $\Re f$ which has $\|\Re f\|_{A(G)} \leq \|f\|_{A(G)}$, $\Re f$ is $\epsilon$-almost integer-valued, and $\supp f_\Z = \supp (\Re f)_\Z$.

Since $f \in A(G)$ there is a Hilbert space $H$ with elements $v,w \in H$ and a homomorphism $\pi:G \rightarrow \Aut(H)$ such that $f(t)=\langle \pi(t)v,w\rangle$ for all $t \in G$, and $\|v\|\|w\|\leq M$.

Let $k$ and $l$ be natural numbers to be optimised later and suppose that $A:=\supp f_\Z$ is \emph{not} $(k,l)$-arithmetically connected.  It follows that there is some $x \in A^k$ such that for every $1 \leq r \leq l$, $i \in N_{k,2r+1}$ and $\sigma \in \{-1,1\}^{2r+1}$ we have
\begin{equation*}
\left|f \left(x_{i_1}^{\sigma_1} \cdots x_{i_{2r+1}}^{\sigma_{2r+1}}\right)\right| \leq \epsilon.
\end{equation*}
We define an auxiliary function $\omega \in (S^1)^k$ as follows.  Since the function $f$ is real and since $x_j \in A$ we see that $|f(x_j)| \geq 1-\epsilon \geq \frac{1}{2}$ for all $j \in [k]$ so that $\sgn f(x_j) \neq 0$. For each $j \in [k]$ if $\sgn f(x_j^{-1})=0$ or $\sgn f(x_j)=\sgn f(x_j^{-1})$ then set $\omega_j=\sgn f(x_j)$ otherwise set $\omega_j=i\sgn f(x_j)$.  It follows that
\begin{align*}
\left|\frac{1}{2k}\sum_{j=1}^k{(\omega_jf(x_j) + \omega_j^{-1}f(x_j^{-1}))}\right|& \geq  \frac{1}{2k}\left(\left|\sum_{j:\omega_j \in \R}^k{|f(x_j)|}\right|^2+\left|\sum_{j:\omega_j \not\in \R}^k{|f(x_j)|}\right|^2 \right)^{\frac{1}{2}} \geq \frac{1}{8}.
\end{align*}
The operator
\begin{equation*}
R:=\frac{1}{2k}\sum_{j=1}^k{\left(\omega_j\pi(x_j) + \omega_j^{-1}\pi(x_j)^*\right)}
\end{equation*}
is hermitian and has eigenvalues in $[-1,1]$.  Let $T_{2l+1}(X)=a_1X+a_3X^3+\cdots + a_{2l+1}X^{2l+1}$ be the Chebychev polynomial (of the first kind) of degree $2l+1$.  Since $T_{2l+1}$ maps $[-1,1]$ to $[-1,1]$, the spectral radius of $T_{2l+1}(R)$ is at most $1$.

Now, for $1 \leq r \leq l$ we have
\begin{align*}
& \left\langle \left(\frac{1}{2k}\sum_{j=1}^k{\left(\omega_j\pi(x_j) + \omega_j^{-1}\pi(x_j)^*\right)}\right)^{2r+1}v,w\right\rangle\\
 & \qquad \qquad= \frac{1}{(2k)^{2r+1}}\sum_{i \in N_{k,2r+1},\sigma \in \{-1,1\}^{2r+1}}{\omega_{i_1}^{\sigma_{1}}\cdots \omega_{i_{2r+1}}^{\sigma_{2r+1}}f \left(x_{i_1}^{\sigma_1} \cdots x_{i_{2r+1}}^{\sigma_{2r+1}}\right)}\\ &\qquad \qquad \qquad \qquad  + \frac{1}{(2k)^{2r+1}}\sum_{i \in T_{k,2r+1},\sigma \in \{-1,1\}^{2r+1}}{\omega_{i_1}^{\sigma_{1}}\cdots \omega_{i_{2r+1}}^{\sigma_{2r+1}}f \left(x_{i_1}^{\sigma_1} \cdots x_{i_{2r+1}}^{\sigma_{2r+1}}\right)},
\end{align*}
from which it follows by Lemma \ref{lem.ct} that
\begin{align*}
\left|\left\langle \left(\frac{1}{2k}\sum_{j=1}^k{\left(\omega_j\pi(x_j) + \omega_j^{-1}\pi(x_j)^*\right)}\right)^{2r+1}v,w\right\rangle\right| & \leq \epsilon + MO\left(\frac{r}{k}\right)^r.
\end{align*}
Since $|a_1|=2l+1$, $|a_{2r+1}| = O(l/r)^{2r+1}$ for $1 \leq r \leq l$ and $\sum_{r=1}^l{|a_{2r+1}|}=\exp(O(l))$ we have
\begin{align*}
M \geq \|v\|\|w\| &\geq \left|\left\langle T_{2l+1}\left(\frac{1}{2k}\sum_{j=1}^k{\left(\omega_j\pi(x_j) +\omega_j^{-1}\pi(x_j)^*\right)}\right)v,w\right\rangle\right|\\
& = \left|\sum_{r=0}^l{a_{2r+1} \left\langle \left(\frac{1}{2k}\sum_{j=1}^k{\left(\omega_j\pi(x_j) + \omega_j^{-1}\pi(x_j)^*\right)}\right)^{2r+1}v,w\right\rangle}\right|\\
& \geq |a_1|\left|\frac{1}{2k}\sum_{j=1}^k{(\omega_jf(x_j)+\omega_j^{-1}f(x_j^{-1}))}\right| - \sum_{r=1}^l{|a_{2r+1}|\left(\epsilon + MO\left(\frac{r}{k}\right)^r\right)}\\
& \geq \frac{(2l+1)}{8} - \epsilon \exp(O(l)) - O(Ml^3/k)
\end{align*}
provided $l^2 \leq ck$ for some sufficiently small absolute $c>0$.  Let $k=c'l^3$ and note that for $l=CM$ sufficiently large we arrive at a contradiction.  The result is proved.
\end{proof}
The next lemma is the part of our argument that depends most on $G$ being finite -- we need some way to measure the size of the set $A$ so that it is non-trivial.
\begin{lemma}\label{lem.s}
Suppose that $A$ is $(k,l)$-arithmetically connected.  Then there is some $S \subset A$ such that $|S| = k^{-O(l)}|A|$ and $|SS^{-1}| \leq k^{O(l)}|S|$.
\end{lemma}
\begin{proof}
Since
\begin{equation*}
\sum_{1 \leq r \leq l}{\sum_{i \in N_{k,2r+1}}{\sum_{\sigma \in\{-1,1\}^{2r+1}}{1}}} = \sum_{1 \leq r \leq l}{|N_{k,2r+1}|\left|\{-1,1\}^{2r+1}\right|} \leq \sum_{r=1}^l{(2k)^{2r+1}} \leq (2k)^{2l+3}
\end{equation*}
we conclude by averaging that there is some $1 \leq r \leq l$, $i \in N_{k,2r+1}$ and $\sigma \in \{-1,1\}^{2r+1}$ such that
\begin{equation*}
\sum_{x_1,\dots,x_k}{1_A(x_1)\dots1_A(x_k)1_A\left(x_{i_1}^{\sigma_1}\cdots x_{i_{2r+1}}^{\sigma_{2r+1}}\right)} \geq \frac{1}{(2k)^{2l+3}}|A|^k.
\end{equation*}
Since $i \in N_{k,2r+1}$ it follows that
\begin{equation*}
|\{j \in [k]: |\{s \in [2r+1]: i_s=j\}|=1\} \geq 2,
\end{equation*}
and so there are elements $s_1<s_2 \in [2r+1]$ such that
\begin{equation}\label{eqn.preimageone}
|\{s \in [2r+1]: i_s=i_{s_1}\}|=1 \text{ and }|\{s \in [2r+1]: i_s=i_{s_2}\}|=1.
\end{equation}
Now
\begin{equation*}
\sum_{\substack{x_{i_1},\dots,x_{i_{s_1-1}},x_{i_{s_1+1}},\dots, \\ x_{i_{s_2-1}},x_{i_{s_2+1}},\dots,x_{i_{2r+1}}}}{1_A\left(x_{i_1}\right)\cdots 1_A\left(x_{i_{s_1-1}}\right)1_A\left(x_{i_{s_1+1}}\right)\cdots 1_A\left(x_{i_{s_2-1}}\right)1_A\left(x_{i_{s_2+1}}\right)\cdots 1_A\left(x_{i_{2r+1}}\right)}
\end{equation*}
equals $|A|^{k-2}$.  Averaging gives $x_{i_1} ,\dots ,x_{i_{s_1-1}} ,x_{i_{s_1+1}} ,\dots ,x_{i_{s_2-1}} ,x_{i_{s_2+1}},\dots,x_{i_{2r+1}} \in A$ such that
\begin{align*}
& \sum_{x_{i_{s_1}},x_{i_{s_2}}}{1_A\left(x_{i_{s_1}}\right)1_A\left(x_{i_{s_2}}\right)}\\
& \qquad \qquad \qquad \times 1_A\left(x_{i_1}^{\sigma_1}\cdots x_{i_{s_1-1}}^{\sigma_{s_1-1}}\cdot x_{i_{s_1}}^{\sigma_{s_1}}\cdot x_{i_{s_1+1}}^{\sigma_{s_1+1}}\cdots x_{i_{s_2-1}}^{\sigma_{s_2-1}}\cdot x_{i_{s_2}}^{\sigma_{s_2}}\cdot x_{i_{s_2+1}}^{\sigma_{s_2+1}}\cdots x_{i_{2r+1}}^{\sigma_{2r+1}}\right)\\ & \qquad \qquad  \qquad \qquad \qquad\geq \frac{1}{|A|^{k-2}}\cdot \frac{1}{(2k)^{2l+3}}|A|^k = \frac{|A|^2}{(2k)^{2l+3}}.
\end{align*}
Write
\begin{equation*}
w:=x_{i_1}^{\sigma_1}\cdots  x_{i_{s_1-1}}^{\sigma_{s_1-1}}, y:=x_{i_{s_1+1}}^{\sigma_{s_1+1}}\cdots x_{i_{s_2-1}}^{\sigma_{s_2-1}} \text{ and } z:= x_{i_{s_2+1}}^{\sigma_{s_2+1}}\cdots x_{i_{2r+1}}^{\sigma_{2r+1}},
\end{equation*}
where each of these products may be empty and if so is the identity.  In light of (\ref{eqn.preimageone}) none of $w,y$ or $z$ depends on $x_{i_{s_1}}$ or $x_{i_{s_2}}$.  By the change of variables
\begin{equation*}
u:=wx_{i_{s_1}}^{\sigma_{s_1}} \text{ and } v^{-1}:=yx_{i_{s_2}}^{\sigma_{s_2}}
\end{equation*}
we have
\begin{align*}
\frac{|A|^2}{(2k)^{2l+3}} & \leq \sum_{u,v}{1_{wA^{\sigma_{s_1}}}(u) 1_{A^{-\sigma_{s_2}}y^{-1}}(v)1_{Az^{-1}}(uv^{-1})}\\
& =  \langle 1_{wA^{\sigma_{s_1}}}, 1_{Az^{-1}} \ast 1_{A^{-\sigma_{s_2}}y^{-1}}\rangle_{\ell_2(G)} \leq \|1_{wA^{\sigma_{s_1}}}\|_{\ell_2(G)} \| 1_{Az^{-1}} \ast 1_{A^{-\sigma_{s_2}}y^{-1}}\|_{\ell_2(G)}.
\end{align*}
It follows that, in the notation of Tao and Vu \cite[(2.37)]{taovu::}, we have $E(Az^{-1},A^{-\sigma_{s_2}}y^{-1}) \geq k^{-O(l)}|A|^3$. Apply the Balog-Szemer{\'e}di-Gowers Lemma \cite[Corollary 2.40]{taovu::} to get $A' \subset Az^{-1}$ such that $|A'| = k^{-O(l)}|A|$ and $|(A')(A')^{-1}| \leq k^{O(l)}|A'|$; the result follows on taking $S:=A'z$ so that $S \subset A$, $|S| = k^{-O(l)}|A|$ and $|SS^{-1}| = |(A'z)(A'z)^{-1}| \leq k^{O(l)}|S|$.
The result is proved.
\end{proof}

\begin{proof}[Proof of Proposition \ref{prop.infstruct}]
This follows immediately on combining Lemma \ref{lem.st} and Lemma \ref{lem.s}.
\end{proof}

\section*{Acknowledgement}

The author should like to thank the referee for thoughtful comments, and identifying an error in the proof of Theorem \ref{thm.ky}.

\bibliographystyle{halpha}

\bibliography{references}

\end{document}